\documentclass[10pt]{article}
\usepackage[T1]{fontenc}
\usepackage{amsmath,amsfonts,amsthm,mathrsfs,amssymb}
\usepackage{cite}
\usepackage{graphicx,float}
\usepackage{subfigure}
\usepackage{placeins}
\usepackage{color}
\usepackage{indentfirst}
\usepackage[bookmarks=true]{hyperref}
\numberwithin{equation}{section}
\newcommand{\R}{{\mathbb R}}
\newcommand{\Z}{{\mathbb Z}}

\newcommand{\C}{{\mathbb C}}

\newcommand{\be}{\begin{eqnarray}}
\newcommand{\ben}{\begin{eqnarray*}}
\newcommand{\en}{\end{eqnarray}}
\newcommand{\enn}{\end{eqnarray*}}

\newcommand{\pa}{\partial}

\newcommand{\ov}{\overline}
\newcommand{\curl}{{\rm curl\,}}

\newcommand{\divv}{{\rm div\,}}

\newcommand{\G}{\Gamma}

\newcommand{\Om}{\Omega}

\newtheorem{theorem}{Theorem}[section]
\newtheorem{lemma}[theorem]{Lemma}

\newtheorem{remark}[theorem]{Remark}

\definecolor{rot}{rgb}{0.000,0.000,0.000}
\definecolor{rot1}{rgb}{0.000,0.000,0.000}
\begin{document}
\renewcommand{\theequation}{\arabic{section}.\arabic{equation}}
\begin{titlepage}
\title{Boundary integral equation methods for the elastic and thermoelastic waves in three dimensions}

\author{Gang Bao\thanks{School of Mathematical Sciences, Zhejiang University, Hangzhou 310027, China. Email: {\tt baog@zju.edu.cn}}\;,
Liwei Xu\thanks{School of Mathematical Sciences, University of Electronic Science and Technology of China, Chengdu, Sichuan 611731, China. Email: {\tt xul@uestc.edu.cn}}\;,
Tao Yin\thanks{Department of Computing and Mathematical Sciences, California Institute of Technology, 1200 East California Blvd., CA 91125, United States. Email:{\tt taoyin89@caltech.edu}}
}

\end{titlepage}
\maketitle
%\vspace{.2in}

\begin{abstract}
In this paper, we consider the boundary integral equation (BIE) method for solving the exterior Neumann boundary value problems of elastic and thermoelastic waves in three dimensions based on the Fredholm integral equations of the first kind. The innovative contribution of this work lies in the proposal of the new regularized formulations for the hyper-singular boundary integral operators (BIO) associated with the time-harmonic elastic and thermoelastic wave equations. With the help of the new regularized formulations, we only need to compute the integrals with weak singularities at most in the corresponding variational forms of the boundary integral equations. The accuracy of the regularized formulations is demonstrated through numerical examples using the Galerkin boundary element method (BEM).

\end{abstract}

\section{Introduction}
\label{sec:i}

In this paper, we apply the BIE method to solve the three dimensional time-harmonic elastic and thermoelastic scattering problems that are of great importance in many fields of applications such as geophysics, seismology, non-destructive testing and material sciences, to name a few. We are interested in the wave scattering by a bounded impenetrable obstacle immersed in an infinite isotropic solid. Based upon the assumption that any deformation of the elastic medium occurs under  a constant temperature, the displacement in the solid can be modeled by the Navier equation together with an appropriate radiation condition at infinity (\cite{KGBB79}). If considering the temperature fluctuations caused by the dynamic deformation, one arrives at the Biot system of equations (\cite{B56,KGBB79,N75})  describing   the interaction of the temperature and displacement fields.

 When considering the wave propagation in an unbounded domain, one could use the so-called Dirichlet-to-Neumann (DtN) map (or non-reflecting boundary condition) on a  closed artificial boundary which decomposes the exterior region into  two parts. After that, the original scattering problem could be solved on the bounded region. The DtN map for the elastic wave has been used for numerical simulations in the open literatures (\cite{GK90,HS98}) and some properties of the DtN map have been investigated in \cite{BHSY,L16,LY}. For the thermoelastic wave, the explicit formulation of the DtN map is still unknown. We refer to \cite{CD98,CD99,DK88,DK90} for the mathematical analysis of the thermoelastic wave. The BIE method is another conventional numerical method for solving the scattering problems, and it has been widely used in acoustics, electromagnetics, elastodynamics and thermoelastodynamics (\cite{BLR14,BXY,C00,CBS08,DB90,HW04,HW08,HX11,JWX14,LH16,L14,LR93,MB88,N01,RS,SS84,TC07,TC09,YHX}). The BIE method takes some advantages over domain discretization methods, including that the boundary integral representation of the solution fulfills the radiation condition naturally,  the dimension of the computational domain is reduced by one, and to name a few. Various numerical techniques, including the Galerkin scheme, the Nystr\"om method, the fast multipole method, and the spectral method etc., have been developed for the efficient transformation of the BIE into a linear system in the past decades. In this work, we will use the Galerkin scheme (\cite{HW04,HW08,RS}) for the numerical solutions, and its advantages include the availability of mathematical convergence analysis allowing $h$-$p$  approximations, and particularly the strength on dealing with the hyper-singularities in the boundary integrals.

During the application of the BIE method, there needs for the use of hyper-singular BIO in many situations,  containing  the removal of the pollution of eigenfrequencies for the BIE (\cite{BM71}), and the solution of the Neumann boundary value problem using the Fredholm integral equation of the first kind (\cite{GN78,N82}), etc.. Theoretical analysis indicates that the hyper-singular BIO is equivalent to the Hadamard finite part of a hyper-singular integral (\cite{HW08}) which is usually difficult to be calculated accurately. One usually needs additional treatments in numerics  for the correct evaluation of  the classically non-integrable boundary integrals arising from the hyper-singular BIO. There are already existing many works (\cite{GN78,LR93,H94,HW08,M49,M66,L14,YHX,BXY}) on this issue, and the main idea consists of rewriting the hyper-singular BIO in terms of a composition of differentiation and weakly singular operators  for the Laplace equation, the Helmholtz equation, the time-harmonic Navier equation and  the Lam\'e equation. This composition, in fact, is a regularization procedure (\cite{HW08, YHX}) of the hyper-singular distribution, and is useful for the variational formulation and related computational procedures.

In this paper, we consider the three dimensional elastic and thermoelastic scattering problems with Neumann boundary conditions. For each problem, we apply the double-layer potential to represent the solution, and then the original boundary value problem is reduced to a Fredholm boundary integral equation of the first kind with the corresponding hyper-singular BIO.  Following the work in \cite{YHX,BXY} for the two-dimensional case, and utilizing the tangential G\"unter derivative, we derive the new and analytically accurate regularized formulations for the hyper-singular BIO  associated with the three dimensional time-harmonic Navier equation and Biot system of linearized thermoelasticity, respectively. As a result, in the corresponding weak forms, all involved integrals are at most weakly-singular. This work is an extension of our work only considering two dimensional elastic waves, and the extension is in fact non-trivial. In the numerical implementations,  applying the special local coordinate system given in \cite{RS} and the Gauss quadrature rules on the triangle element, we present a semi-analytic strategy to evaluate all the weakly-singular integrals effectively. Although  we only consider $C^2$ boundary in our numerical tests, the theoretical results actually could be extended to the Lipschitz case in terms of the properties of G\"unter derivative given in \cite{BT}. The convergence analysis of the numerical scheme could be obtained following the standard techniques in \cite{HW08}, and we will omit it in this work.

The rest of this paper is organized as follows. In Section \ref{sec:ge}, we introduce the exterior elastic and thermoelastic scattering problems,  and then describe the BIE and the Galerkin BEM in Section \ref{sec:nm}. In section \ref{sec:rf}, we propose the new and analytically accurate regularized formulations for the hyper-singular BIO  in three dimensions. Finally, we discuss a semi-analytic strategy for the numerical implementation of the Galerkin scheme and present numerical results of several examples in section \ref{sec:ne}.

\section{Mathematical problems}
\label{sec:ge}
Let $\Omega\subset \R^3$ be a bounded, simply connected and impenetrable body with $C^2$ boundary $\Gamma=\partial\Omega$. The exterior complement of $\Omega$ is denoted by $\Omega^c= \R^3\setminus\overline{\Omega}\subset \R^3$. Assume that $\Omega^c$ is occupied by a linear and isotropic elastic solid characterized by the Lam\'e constants $\lambda$ and $\mu$ ($\mu>0$, $3\lambda+2\mu>0$) and mass density $\rho>0$. Let $\omega>0$ be the frequency of propagating waves.
\subsection{Elastic scattering problem (ESP)}

Assume that the temperature is always a constant and suppress the time-harmonic dependence $e^{-i\omega t}$. Then the displacement field $u$ in the solid can be modeled by the following exterior ESP: Given $f\in H^{-1/2}(\Gamma)^3$, determine $u=(u_1,u_2,u_3)^\top\in H_{loc}^1(\Omega^c)^3$ satisfying
\be
\label{Navier}
\Delta^{*}u +   \rho \omega^2u &=& 0\quad\mbox{in}\quad \Omega^c,\\
\label{BoundCond}
T(\pa,\nu)u &=& f \quad \mbox{on}\quad \Gamma,
\en
and the Kupradze radiation condition (\cite{KGBB79})
\be
\label{RadiationCond}
\lim_{r \to \infty} r\left(\frac{\partial u_t}{\partial r}-ik_tu_t\right) = 0,\quad r=|x|,\quad t=p,s,
\en
uniformly with respect to all $\hat{x}=x/|x|\in\mathbb{S}^2:=\{x\in\R^3:|x|=1\}$. Here, $\Delta^{*}$ is the Lam\'e operator defined by
\ben
\label{LameOper}
\Delta^* := \mu\,\mbox{div}\,\mbox{grad} + (\lambda + \mu)\,\mbox{grad}\, \mbox{div}\,,
\enn
and the traction operator $T(\pa,\nu)$ on the boundary is defined as
\ben
\label{stress-3D}
T(\pa,\nu)u:=2 \mu \, \partial_{\nu} u + \lambda \,
\nu \, \divv u+\mu \nu\times \curl u,\quad \nu=(\nu^1,\nu^2,\nu^3){^\top},
\enn
where $\nu$ is the outward unit normal to the boundary $\G$ and $\partial_\nu:=\nu\cdot\mbox{grad}$ is the normal derivative. In (\ref{RadiationCond}), $u_p$ and $u_s$ are referred as the compressional wave and the shear wave, respectively, and they are given by
\ben
u_p=-\frac{1}{k_p^2}\,\mbox{grad}\,\mbox{div}\;u,\quad u_s=\frac{1}{k_s^2}\,\mbox{curl}\,\mbox{curl}\;u,
\enn
where the wave numbers $k_s,k_p$ are defined as
\ben
k_s := \omega/c_p,\quad k_p := \omega/c_s,
\enn
with
\ben
c_p:=\sqrt{\mu/\rho},\quad c_s:=\sqrt{(\lambda+2\mu)/\rho}.
\enn

For the uniqueness of the ESP (\ref{Navier})-(\ref{RadiationCond}), we refer to \cite{KGBB79,BHSY}.

\subsection{Thermoelastic scattering problem (TESP)}

Now we consider the temperature fluctuations caused by the dynamic deformation. In this case,  the elastic medium $\Omega^c$ is additionally characterized by the coefficient of thermal diffusity $\kappa$ and the coupling constants $\gamma$, $\eta$ given by
\ben
\gamma=(\lambda+\frac{2}{3}\mu)\alpha,\quad \eta=\frac{T_0\gamma}{\lambda_0},
\enn
respectively,  where $\alpha$ is the volumetric thermal expansion coefficient, $T_0$ is a reference value of the absolute temperature and $\lambda_0$ is the coefficient of thermal conductivity. Denote by $\epsilon:=\gamma\eta\kappa/(\lambda+2\mu)$ the dimensionless thermoelastic coupling constant which assumes 'small' positive values for most thermoelastic media and $q=i\omega/\kappa$. Suppressing the time-harmonic dependence $e^{-i\omega t}$, the displacement field $u$ and the temperature variation field $p$ can be modeled by the following Biot system of linearized thermoelasticity
\be
\label{thermo1}
\Delta^{*}u+\rho\omega^2u-\gamma\nabla p &=& 0\quad\mbox{in}\quad \Omega^c,\\
\label{thermo2}
\Delta p+qp+i\omega\eta\nabla\cdot u &=& 0 \quad\mbox{in}\quad \Omega^c.
\en
Rewriting (\ref{thermo1})-(\ref{thermo2}) into a matrix form, we obtain
\be
\label{thermo12}
LU=0,\quad L:=\begin{bmatrix}
(\mu\Delta+\rho\omega^2)I_3+(\lambda+\mu)\nabla\nabla\cdot & -\gamma\nabla\\
q\eta\kappa\nabla\cdot & \Delta+q
\end{bmatrix},\quad U=(u^\top,p)^\top.
\en
On the boundary of the scatterer, we assume the Neumann boundary condition
\be
\label{BC}
\widetilde{T}(\pa,\nu)U:=\begin{bmatrix}
T(\pa,\nu) & -\gamma\nu \\
0 & \pa_\nu
\end{bmatrix}U=F.
\en
It follows (\cite{KGBB79}) that the wave field $U$ admits the decomposition  \ben
u=u^1+u^2+u^s,\quad p=p^1+p^2,
\enn
where the vector displacement fields $u^1,u^2,u^s$ satisfy the vectorial Helmholtz equations
\ben
\Delta u^i+k_i^2u^i=0,\quad i=1,2 \quad\mbox{and}\quad \Delta u^s+k_s^2u^s=0
\enn
with
\ben
\mbox{curl}\,u^i=0\quad i=1,2 \quad\mbox{and}\quad \mbox{div}\,u^s=0,
\enn
and the scalar temperature fields $p^1$ and $p^2$ satisfy the following scalar Helmholtz equations
\ben
\Delta p^i+k_i^2p^i=0,\quad i=1,2.
\enn
Here, the wave numbers $k_1,k_2$, corresponding to the elastothermal and thermoelastic waves  respectively,  are the roots of the characteristic system
\be
\label{k12}
k_1^2+k_2^2=q(1+\epsilon)+k_p^2,\quad k_1^2k_2^2=qk_p^2,
\en
for which $\mbox{Im}\,k_i\ge 0, i=1,2$. In particular,
\ben
k_1 &=& \frac{1}{2c_p}\sqrt{\frac{\omega}{\kappa}}\left[\sqrt{\omega\kappa +C_+} + \sqrt{\omega\kappa +C_-}\right],\\
k_2 &=& \frac{1}{2c_p}\sqrt{\frac{\omega}{\kappa}}\left[\sqrt{\omega\kappa +C_+} - \sqrt{\omega\kappa +C_-}\right],
\enn
where
\ben
C_\pm=i(1+\epsilon)c_p^2\pm (1+i)c_p\sqrt{2\omega\kappa}.
\enn
We assume that the scattered field $U$ satisfies the following Kupradze radiation conditions as $r=|x|\rightarrow \infty$ for $i=1,2,3$ and $j=1,2$
\ben
&& u^j=o(r^{-1}),\quad \pa_{x_i}u^j=O(r^{-2}),\\
&& p^j=o(r^{-1}),\quad \pa_{x_i}p^j=O(r^{-2}),\\
&& u^s=o(r^{-1}),\quad r(\pa_ru^s-ik_su^s)=O(r^{-1}).
\enn

The direct TESP to be considered in this paper is to determine  the displacement field $u$ and the temperature variation field $p$ satisfying (\ref{thermo12}), the boundary condition (\ref{BC}) and the Kupradze radiation conditions. For given $F\in (H^{-1/2}(\G))^4$, we refer to \cite{KGBB79,C00} for the uniqueness of the direct problem.

\section{Numerical method}
\label{sec:nm}
In this section, we derive the BIE  for solving the ESP and TESP, respectively and give a brief introduction to the Galerkin BEM for the discretization of the derived BIE.

\subsection{BIE for ESP}
\label{sec:bie1}

For the ESP, it follows from the potential theory (\cite{KGBB79}) that the unknown function $u$ can be represented as
\be
\label{DirectBRF1}
u(x)  = (\mathcal{D}_s\varphi)(x):=  \int_{\Gamma}(T(\pa_y,\nu_y)  E(x,y))^\top \varphi(y)\,ds_y, \quad \forall\,x\in\Omega^c,
\en
where $\mathcal{D}_s$ is referred as the double-layer potential and $E(x,y)$ is the fundamental displacement tensor of the time-harmonic Navier equation (\ref{Navier}) in $\R^3$ taking the form
\be
\label{NavierFS}
E(x,y) = \frac{1}{\mu}\gamma_{k_s}(x,y) I + \frac{1}{\rho\omega^2}\nabla_x\nabla_x^\top \left[\gamma_{k_s}(x,y) - \gamma_{k_p}(x,y)\right],\quad x\ne y.
\en
In (\ref{NavierFS}) and the following, $I$ denotes the $3\times 3$ identity matrix, and $\gamma_{k_t}(x,y)$ is the fundamental solution of the Helmholtz equation in $\R^3$ with wave number $k_t$ and takes the form
\be
\label{HelmholtzFS}
\gamma_{k_t}(x,y) =\frac{\exp(ik_t|x-y|)}{4\pi|x-y|},
\quad x\ne y,\quad t=p,s.
\en
Operating with the traction operator on (\ref{DirectBRF1}), taking the limits as $x\to\G$ and applying the jump relations and the boundary condition (\ref{BoundCond}) , we arrive at the BIE on $\G$
\be
\label{BIE1}
W_s\varphi(x)=-f,\quad x\in\Gamma,
\en
where $W_s:H^{s}(\Gamma)^3\rightarrow H^{s-1}(\Gamma)^3(s\ge 1/2)$ is called the hyper-singular BIO defined by
\be
\label{HBIO1}
W_su(x)  := -\lim_{z\rightarrow x\in\G,z\notin\G}T(\pa_z,\nu_x)\int_{\Gamma}(T(\pa_y,\nu_y)E(z,y))^\top u(y)\,ds_y,\quad x\in\Gamma.
\en
The standard weak formulation of (\ref{BIE1}) reads: Given $f\in H^{-1/2}(\Gamma)^3$, find $\varphi\in H^{1/2}(\G)^3$ such that
\be
\label{weak1}
\langle W_s\varphi,v\rangle=-\langle f,v\rangle \quad\mbox{for all}\quad v\in H^{1/2}(\G)^3.
\en
Here and in the sequel, $\langle\cdot,\cdot\rangle$ denotes the $L^2$ duality pairing between $H^{-1/2}(\G)^d$ and $H^{1/2}(\G)^d$ for $d\in\Z^+$.

\subsection{BIE for TESP}

For the TESP, it follows from the potential theory (\cite{KGBB79,C00}) that the unknown function $U$ can be represented as
\be
\label{DirectBRF2}
U(x)  = (\widetilde{\mathcal{D}}\Psi)(x):=  \int_{\Gamma}(\widetilde{T}^*(\pa_y,\nu_y) \widetilde{E}^\top(x,y))^\top \Psi(y)\,ds_y, \quad \forall\,x\in\Omega^c,
\en
where $\widetilde{E}(x,y)$ is the fundamental solution of the Biot system (\ref{thermo12}) in $\R^3$ given by
\ben
\widetilde{E}(x,y)=\begin{bmatrix}
E_{11}(x,y) & E_{12}(x,y) \\
E_{21}^\top(x,y) & E_{22}(x,y)
\end{bmatrix},
\enn
with
\ben
E_{11}(x,y)&=& \frac{1}{\mu}\gamma_{k_s}(x,y)I+\frac{1}{\rho\omega^2} \nabla_x\nabla_x^\top \left[\gamma_{k_s}(x,y)-\gamma_{k_1}(x,y) \right]\\
&-& \frac{1}{\rho\omega^2}\frac{k_p^2-k_1^2}{k_1^2-k_2^2} \nabla_x\nabla_x^\top \left[\gamma_{k_1}(x,y)-\gamma_{k_2}(x,y) \right],\\
E_{12}(x,y)&=& -\frac{\gamma}{(k_1^2-k_2^2)(\lambda+2\mu)}\nabla_x \left[ \gamma_{k_1}(x,y)-\gamma_{k_2}(x,y)\right],\\
E_{21}(x,y)&=& \frac{i\omega\eta}{(k_1^2-k_2^2)(\lambda+2\mu)}\nabla_x \left[ \gamma_{k_1}(x,y)-\gamma_{k_2}(x,y)\right],\\
E_{22}(x,y)&=& -\frac{1}{k_1^2-k_2^2}\left[ (k_p^2-k_1^2)\gamma_{k_1}(x,y)-(k_p^2-k_2^2)\gamma_{k_2}(x,y) \right],
\enn
and $\widetilde{T}^*(\pa,\nu)$ is the corresponding Neumann operator of the adjoint problem of (\ref{thermo12}) taking the form
\ben
\widetilde{T}^*(\pa,\nu):=\begin{bmatrix}
T(\pa,\nu) & -i\omega\eta\nu \\
0 & \pa_\nu
\end{bmatrix}.
\enn
Operating with the operator $\widetilde{T}$ on (\ref{DirectBRF2}), taking the limits as $x\to\G$ and applying the boundary condition (\ref{BC}) , we obtain the BIE on $\G$
\be
\label{BIE2}
\widetilde{W}_s\Psi(x)=-F,\quad x\in\G,
\en
where the hyper-singular BIO $\widetilde{W}_s$ is defined by
\be
\label{HBIO2}
\widetilde{W}_s\Psi(x) := -\lim_{z\rightarrow x\in\G,z\notin\G}\widetilde{T}(\pa_z,\nu_x)\int_\G (\widetilde{T}^*(\pa_y,\nu_y)E^\top(z,y))^\top\Psi(y)ds_y,\quad x\in\G.
\en
The standard weak formulation of (\ref{BIE2}) reads: Given $F\in H^{-1/2}(\Gamma)^4$, find $\Psi\in H^{1/2}(\G)^4$ such that
\be
\label{weak2}
\langle\widetilde{W}_s\Psi,V\rangle=-\langle F,V\rangle \quad\mbox{for all}\quad V\in H^{1/2}(\G)^4.
\en

\begin{remark}
For the wellposedness of the variational equations (\ref{weak1}) and (\ref{weak2}), we refer the readers to \cite{BXY,C00,KGBB79}. The pollution of eigenfrequencies on the uniqueness can be removed by applying the so-called Burton-Miller formulation, see \cite{BXY,BM71} for example.
\end{remark}

\subsection{Galerkin BEM}
\label{sec:gbem}

We only propose the Galerkin Scheme for solving ESP, and the corresponding procedure and formulas for TESP are quite similar and will be neglected.
Let $\mathcal{H}_h$ be a finite dimensional subspace of $H^{1/2}(\G)$. Then the Galerkin approximation of (\ref{BIE1}) reads: Given $f$, find $\varphi_h\in\mathcal{H}_h^3$ satisfying
\be
\label{galerkin}
\langle W_s\varphi_h,v_h\rangle=-\langle f,v_h\rangle \quad\mbox{for all}\quad v_h\in\mathcal{H}_h^3.
\en
In the following, we describe briefly the reduction of the Galerkin equation (\ref{galerkin}) into its discrete linear system of equations.

Let $\G_h=\cup_{i=1}^N \ov{\tau_i}$ be a uniform boundary element mesh of $\G$ where each $\tau_i$ is a plane triangle with vertex $x_{i_1},x_{i_2},x_{i_3}$ ordered counter clockwise. Let $\{x_j\}_{j=1}^M$ be the set of all nodes of the triangulation. Using the reference element
\ben
\tau_\xi=\{\xi=(\xi_1,\xi_2)^\top\in\R^2,0<\xi_1<1,\; 0<\xi_2<1-\xi_1\},
\enn
the point $x\in\tau_i$ can be parameterized as
\ben
x=x(\xi)=x_{i_1}+\xi_1(x_{i_2}-x_{i_1})+\xi_2(x_{i_3}-x_{i_1}),\quad \xi\in\tau_\xi.
\enn
Let $\{\psi_j\}_{j=1}^M$ be the set of piecewise linear basis functions. We seek the approximate solution
\ben
u_h(x)=\sum_{j=1}^M u_j\psi_j(x),
\enn
where $u_j\in\C^3, j=1,\cdots,M$ are unknown nodal values of $u_h$ at $x_j$. For the boundary value $f$, we interpolated it as
\ben
f_h=\sum_{i=1}^N f(x_{i_*})\phi_i(x),
\enn
where $x_{i_*}=(x_{i_1}+x_{i_2}+x_{i_3})/3$ is the mid point of the element $\tau_i$ and $\{\phi_i\}_{i=1}^N$ is the set of piecewise constant basis functions. Substituting the interpolation forms into (\ref{galerkin}) and setting $\psi_j, j=1,\cdots,M$ as test functions, we arrive at a linear system of equations
\be
\label{linearsys}
{\bf A}_h {\bf X}={\bf B}_h{\bf b}, {\bf A}_h\in\C^{3M\times 3M}, {\bf B}_h\in\C^{3M\times 3N}, {\bf b}=({\bf b}_1^\top,\cdots,{\bf b}_N^\top)^\top\in\C^{3N\times 1},
\en
where for $k,j=1,\cdots,M$, $i=1,\cdots,N$,
\be
\label{coematrix1}
{\bf A}_h(k,j)= -\int_{\G_h}\left[\lim_{z\rightarrow x\in\G,z\notin\G}T(\pa_z,\nu_x)\int_{\G_h}(T(\pa_y,\nu_y)E(z,y))^\top \psi_j(y)ds_y\right]\psi_k(x)ds_x,
\en
\be
\label{coematrix2}
{\bf B}_h(k,i)=  \int_{\G_h}\phi_i(x)\psi_k(x)\,ds_x\,I\in\C^{3\times3},
\en
\be
\label{coedata}
{\bf b}_i = f(x_{i_*})\in\C^{3\times1}.
\en

\section{Regularized formulation for the hyper-singular BIO}
\label{sec:rf}

In this section, we derive the new regularized formulations for the hyper-singular BIOs $W_s$ and $\widetilde{W}_s$ in three dimensions such that the coefficient matrix ${\bf A}_h$ ( or $\widetilde{\bf A}_h$) can be  evaluated in a more effective and accurate way. More precisely, using the derived regularized formulations, only classically  integrable and weakly-singular integrals are involved in the weak forms of $W_s$ and $\widetilde{W}_s$.  Before doing this, we introduce the hyper-singular BIO associated with the Helmholtz equation and the G\"unter derivatives.

\subsection{Hyper-singular BIO for acoustic scattering problem}
\label{sec:hbioa}

Consider the Helmholtz equation
\ben
\Delta p+k^2p=0\quad\mbox{in}\quad\Omega^c,
\enn
with wave number $k>0$. Denote by $V_f:H^{s-1}(\Gamma)\rightarrow H^s(\Gamma)$ and $W_f:H^s(\Gamma)\rightarrow H^{s-1}(\Gamma)$, $s\ge1/2$ the single-layer and hyper-singular BIO  defined by
\ben
V_f\psi(x) &:=& \int_{\Gamma}\gamma_k(x,y) \psi(y)\,ds_y,\quad x\in\Gamma,\\
W_f\varphi(x)  &:=& -\lim_{z\rightarrow x\in\G,z\notin\G}\nu_x\cdot\nabla_z\int_{\Gamma}\pa_{\nu_y}\gamma_k(z,y) \varphi(y)\,ds_y,\quad x\in\Gamma,
\enn
respectively. It follows from Lemma 1.2.2 in \cite{HW08} that the hyper-singular BIO $W_f$ can be expressed as
\be
\label{Wf2}
W_{f} p(x)=-(\nu_x\times\nabla_x)\cdot V_f(\nu\times\nabla p)(x)-k^2\nu_x^\top V_f(p\nu)(x).
\en

\subsection{G\"unter derivatives}
\label{sec:gd}

Now we describe the G\"unter derivatives that play essential roles in the proof of our main results. Define the operator $M(\pa,\nu)$, whose elements are also called the G\"unter derivatives, as
\ben
M(\pa,\nu)u(x)= \pa_{\nu}u -\nu(\nabla\cdot u)+\nu\times \,\curl\,u.
\enn
Then the traction operator can be rewritten as
\be \label{Tform2}
T(\pa,\nu)u(x)= (\lambda+\mu)\nu(\nabla \cdot u) + \mu\pa_{\nu}u + \mu M(\pa,\nu)u.
\en
A direct calculation yields
\ben
T(\pa,\nu)\nabla &=& (\lambda+\mu)\nu\Delta+\mu\pa_{\nu}\nabla +\mu M(\pa,\nu)\nabla\\
&=& (\lambda+\mu)\nu\Delta+\mu\pa_{\nu}\nabla -\mu M(\pa,\nu)\nabla+ 2\mu M(\pa,\nu)\nabla.
\enn
Then
\ben
\pa_{\nu}\nabla-M(\pa,\nu)\nabla = \nu\Delta,
\enn
which implies that
\be
\label{Tgrad}
T(\pa,\nu)\nabla=(\lambda+2\mu)\nu\Delta+2\mu M(\pa,\nu)\nabla.
\en

The properties of the operator $M(\pa,\nu)$ (\cite{BT,HW08,KGBB79}) shows that for any scaler fields $p,q$, vector fields $u,v$ and tensor field $E$, there hold the Stokes formulas
\be
\label{stokes1}
\int_\Gamma (m^{ij}p)\,q\,ds &=& -\int_\Gamma p\,(m^{ij}q)\,ds,\\
\label{stokes2}
\int_\Gamma (Mu)\cdot v\,ds &=& \int_\Gamma u\cdot (Mv)\,ds,\\
\label{stokes3}
\int_\Gamma (Mq)\,v\,ds &=& -\int_\Gamma q\,(Mv)\,ds,
\en
and
\be
\label{stokes4}
\int_\Gamma (ME)^\top\,vds=\int_\Gamma E^\top\,(Mv)\,ds.
\en

\subsection{Hyper-singular BIO for ESP}
\label{sec:hbioe}

We now investigate the operator $W_s$. Following the results in \cite{L14,YHX}, we have for $x\ne y$,
\be \label{TxExy}
T(\pa_x,\nu_x)E(x,y) &=& - \nu_x \nabla _x^\top [\gamma_{k_s}(x,y)-\gamma_{k_p}(x,y)] + \pa_{\nu_x}\gamma _{k_s}(x,y)I\nonumber\\
&+& M_x\left[2\mu E(x,y) - \gamma
_{k_s}(x,y)I\right],
\en
and
\be \label{TyExy}
T(\pa_y,\nu_y)E(x,y) &=& - \nu_y \nabla _y^\top [\gamma_{k_s}(x,y)-\gamma_{k_p}(x,y)] + \pa_{\nu_y}\gamma _{k_s}(x,y)I\nonumber\\
&+& M_y\left[2\mu E(x,y) - \gamma
_{k_s}(x,y)I\right].
\en

Using the G\"unter derivatives, the hyper-singular operator $W_s$ can be rewritten as
\ben
&\quad& W_su(x)\\
&=& -\lim_{z\rightarrow x\in\G,z\notin\G} \mu\nu_x\cdot\nabla_z \mathcal{D}_su(z) +(\lambda+\mu)\nu_x(\nabla_z \cdot \mathcal{D}_s u(z))+\mu M_{z,x}\mathcal{D}_s u(z),
\enn
where
\ben
M_{z,x}\psi(z)= \pa_{\nu_x}\psi -\nu_x(\nabla_z \cdot \psi)+\nu_x\times \,\curl_z\,\psi, M_{z,x}=[m_{z,x}^{ij}]_{i,j=1}^3.
\enn

\begin{theorem}
\label{main}
The hyper-singular BIO $W_s$ in three dimensions can be expressed alternatively as
\be
\label{Ws1}
W_s u(x) &=& \rho\omega^2\int_\Gamma\left[ \gamma_{k_s}(x,y)(\nu_x\nu_y^\top-\nu_x^\top\nu_yI-J_{\nu_x,\nu_y})- \gamma_{k_p}(x,y)\nu_x\nu_y^\top\right]u(y)ds_y\nonumber\\
&+& 2\mu\int_\Gamma M_x\nabla_y[\gamma_{k_s}(x,y)-\gamma_{k_p}(x,y)]\nu_y^\top u(y)ds_y \nonumber\\
&+& 2\mu\int_\Gamma \nu_x\nabla_x^\top [\gamma_{k_s}(x,y)-\gamma_{k_p}(x,y)] M_yu(y)ds_y \nonumber\\
&-& \mu\int_\Gamma \left(\nu_x\times\nabla_x\gamma_{k_s}(x,y)\right)\cdot \left(\nu_y\times\nabla_yu(y)\right)ds_y\nonumber\\
&+& 2\mu\int_\Gamma M_x\gamma_{k_s}(x,y) M_yu(y)ds_y- 4\mu^2\int_\Gamma M_xE(x,y) M_yu(y)ds_y\nonumber\\
&-& \mu\left\{ \sum_{k,l=1}^3\int_\Gamma m_x^{kl}\gamma_{k_s}(x,y) m_y^{kj}u_l(y)ds_y \right\}_{j=1}^3,
\en
where $J_{\nu_x,\nu_y}=\nu_y\nu_x^\top-\nu_x\nu_y^\top$. %Moreover,
%\be
%\label{Ws2}
%&\quad&\langle W_s u,v\rangle \nonumber\\
%&=& \rho\omega^2\int_\Gamma\int_\Gamma\left[ \gamma_{k_s}(x,y)(\nu_x\nu_y^\top-\nu_x^\top\nu_yI-J_{\nu_x,\nu_y})u(y)\right] \cdot\ov{v}(x)ds_yds_x\nonumber\\
%&-& \rho\omega^2\int_\Gamma\int_\Gamma\left[\gamma_{k_p}(x,y)\nu_x\nu_y^\top u(y)\right] \cdot\ov{v}(x)ds_yds_x\nonumber\\
%&+& 2\mu\int_\Gamma\int_\Gamma \left[\nabla_y[\gamma_{k_s}(x,y)-\gamma_{k_p}(x,y)]\nu_y^\top u(y)\right]\cdot M_x\ov{v}(x)ds_yds_x\nonumber\\
%&+& 2\mu\int_\Gamma\int_\Gamma \left[\nu_x\nabla_x^\top [\gamma_{k_s}(x,y)-\gamma_{k_p}(x,y)] M_yu(y)\right] \cdot\ov{v}(x)ds_yds_x\nonumber\\
%&+& \mu\sum_{l=1}^3\int_\Gamma\int_\Gamma \gamma_{k_s}(x,y) \left(\nu_y\times\nabla_yu_l(y)\right)\cdot \left(\nu_x\times\nabla_x\ov{v}_l(x)\right)ds_yds_x\nonumber\\
%&+& 2\mu\int_\Gamma\int_\Gamma\gamma_{k_s}(x,y) M_yu(y)\cdot M_x\ov{v}(x)ds_yds_x\nonumber\\
%&-& 4\mu^2\int_\Gamma\int_\Gamma \left[E(x,y)M_yu(y)\right] \cdot M_x\ov{v}(x)ds_yds_x\nonumber\\
%&-& \mu \sum_{j,k,l=1}^3\int_\Gamma\int_\Gamma \gamma_{k_s}(x,y) m_y^{kj}u_l(y)m_x^{lk}\ov{v}_j(x)ds_yds_x,
%\en
%in which all the integrals are at most weakly-singular.
\end{theorem}
\begin{proof}
See \ref{appendex.a}.
\end{proof}

\subsection{Hyper-singular BIO for TESP}

Now we consider the operator $\widetilde{W}_s$. Note that the hyper-singular kernel of $\widetilde{W}_s$ is
\ben
\begin{bmatrix}
W_{11}(x,y;z) & W_{12}(x,y;z) \\
W_{21}^\top(x,y;z) & W_{22}(x,y;z)
\end{bmatrix}
\enn
where
\ben
W_{11}(x,y;z) &=& T(\pa_z,\nu_x)(T(\pa_y,\nu_y)E_{11}(z,y))^\top- i\omega\eta T(\pa_z,\nu_x)E_{12}(z,y)\nu_y^\top \\
&-& \gamma\nu_x(T(\pa_y,\nu_y)E_{21}(z,y))^\top +i\omega\eta\gamma\nu_x\nu_y^\top E_{22}(z,y),\\
W_{12}(x,y;z) &=& T(\pa_z,\nu_x)\pa_{\nu_y}E_{12}(z,y)-\gamma\nu_x\pa_{\nu_y}E_{22}(z,y),\\
W_{21}(x,y;z) &=& (\nu_x\cdot\pa_z)T(\pa_y,\nu_y)E_{21}(z,y)-i\omega\eta\nu_y(\nu_x\cdot\pa_z)E_{22}(z,y),\\
W_{22}(x,y;z) &=& (\nu_x\cdot\pa_z)\pa_{\nu_y}E_{22}(z,y).
\enn
For $U=(u^\top,p)$ and $V=(v^\top,q)$, we have
\be
\label{weakW}
\widetilde{W}_sU &=& [W_1u+ W_2p;  W_3u+W_4p]
\en
where
\ben
W_1u(x) &=& -\lim_{z\rightarrow x\in\G,z\notin\G}\int_\G W_{11}(x,y;z)u(y)ds_y,\\
W_2p(x) &=& -\lim_{z\rightarrow x\in\G,z\notin\G}\int_\G W_{12}(x,y;z)p(y)ds_y,\\
W_3u(x) &=& -\lim_{z\rightarrow x\in\G,z\notin\G}\int_\G W_{21}^\top(x,y;z)u(y)ds_y,\\
W_4p(x) &=& -\lim_{z\rightarrow x\in\G,z\notin\G}\int_\G W_{22}(x,y;z)p(y)ds_y.
\enn

\begin{lemma}
\label{Tlemma1}
For $x\ne y$, it follows that
\be
\label{TxE11}
&\quad& T(\pa_x,\nu_x)E_{11}(x,y)\nonumber\\
&=&-\nu_x\nabla_{x}^\top [\gamma_{k_s}(x,y)-\gamma_{k_1}(x,y)] + \frac{k_2^2-q}{k_1^2-k_2^2} \nu_x\nabla_{x}^\top[\gamma_{k_1}(x,y)-\gamma_{k_2}(x,y)] \nonumber\\
&+& \partial_{\nu_x}\gamma_{k_s}(x,y)I+ M_x[2\mu E_{11}(x,y)-\gamma_{k_s}(x,y)I],
\en
and
\be
\label{TyE11}
&\quad& T(\pa_y,\nu_y)E_{11}(x,y)\nonumber\\
&=&-\nu_y\nabla_{y}^\top [\gamma_{k_s}(x,y)-\gamma_{k_1}(x,y)]+ \frac{k_2^2-q}{k_1^2-k_2^2} \nu_y\nabla_{y}^\top[\gamma_{k_1}(x,y)-\gamma_{k_2}(x,y)] \nonumber\\
&+& \partial_{\nu_y}\gamma_{k_s}(x,y)I+ M_y[2\mu E_{11}(x,y)-\gamma_{k_s}(x,y)I].
\en
\end{lemma}
\begin{proof}
See \ref{appendex.b}.
\end{proof}

We first investigate the term $W_1u$. Observe that the first term in $W_{11}$ is consistent with the the kernel of hyper-singular BIO $W_s$. For $W_1u$, we have the following regularized formulation.

\begin{theorem}
\label{main1}
The hyper-singular operator $W_1$ can be expressed alternatively as
\be
\label{W1-1}
W_1 u(x) &=&\rho\omega^2 \int_\G \gamma_{k_s}(x,y) (\nu_x\nu_y^\top-\nu_x^\top\nu_yI-J_{\nu_x,\nu_y}) u(y)ds_y \nonumber\\
&+& \int_\G \left[C_1\gamma_{k_1}(x,y)-C_2\gamma_{k_1}(x,y) \right]\nu_x\nu_y^\top u(y)ds_y \nonumber\\
&-& \mu\left\{ \sum_{k,l=1}^3\int_\G m_x^{kl}\gamma_{k_s}(x,y) m_y^{kj}u_l(y)ds_y \right\}_{j=1}^3 \nonumber\\
&-& \mu\int_\G \left(\nu_x\times\nabla_x\gamma_{k_s}(x,y)\right)\cdot \left(\nu_y\times\nabla_yu(y)\right)ds_y\nonumber\\
&+& 2\mu\int_\G M_x\gamma_{k_s}(x,y) M_yu(y)ds_y -4\mu^2\int_\G M_xE_{11}(x,y) M_yu(y)ds_y\nonumber\\
&+& 2\mu\int_\G \nu_x\nabla_x^\top \left[\gamma_{k_s}(x,y) -\gamma_{k_1}(x,y)\right] M_yu(y)ds_y \nonumber\\
&+& 2\mu\int_\G M_x\nabla_y \left[\gamma_{k_s}(x,y) -\gamma_{k_1}(x,y)\right]\nu_y^\top u(y)ds_y \nonumber\\
&+& C_3\int_\G \nu_x\nabla_x^\top \left[\gamma_{k_1}(x,y) -\gamma_{k_2}(x,y)\right] M_yu(y)ds_y  \nonumber\\
&+& C_3\int_\G M_x\nabla_y \left[\gamma_{k_1}(x,y) -\gamma_{k_2}(x,y)\right]\nu_y^\top u(y)ds_y.
\en
%Moreover,
%\be
%\label{W1-2}
%\langle W_1 u,v\rangle &=& \rho\omega^2\int_\G\int_\G\left[ \gamma_{k_s}(x,y)(\nu_x\nu_y^\top-\nu_x^\top\nu_yI-J_{\nu_x,\nu_y})u(y)\right] \cdot\ov{v}(x)ds_yds_x\nonumber\\
%&+& \int_\G\int_\G \left\{\left[C_1\gamma_{k_1}(x,y)-C_2\gamma_{k_2}(x,y)\right] \nu_x\nu_y^\top u(y)\right\} \cdot\ov{v}(x)ds_yds_x \nonumber\\
%&-& \mu \sum_{j,k,l=1}^3\int_\G\int_\G \gamma_{k_s}(x,y) m_y^{kj}u_l(y)m_x^{lk}\ov{v}_j(x)ds_yds_x \nonumber\\
%&+& \mu\sum_{l=1}^3\int_\G\int_\G \gamma_{k_s}(x,y) \left(\nu_y\times\nabla_yu_l(y)\right)\cdot \left(\nu_x\times\nabla_x\ov{v}_l(x)\right)ds_yds_x\nonumber\\
%&+& 2\mu\int_\G\int_\G\gamma_{k_s}(x,y) M_yu(y)\cdot M_x\ov{v}(x)ds_yds_x\nonumber\\
%&-& 4\mu^2\int_\G\int_\G \left[E_{11}(x,y)M_yu(y)\right] \cdot M_x\ov{v}(x)ds_yds_x\nonumber\\
%&+& 2\mu\int_\G\int_\G \left\{\nu_x\nabla_x^\top \left[\gamma_{k_s}(x,y) -\gamma_{k_1}(x,y)\right] M_yu(y)\right\} \cdot\ov{v}(x)ds_yds_x\nonumber\\
%&+& 2\mu\int_\G\int_\G \left\{\nabla_y\left[\gamma_{k_s}(x,y) -\gamma_{k_1}(x,y)\right]\nu_y^\top u(y)\right\}\cdot M_x\ov{v}(x)ds_yds_x\nonumber\\
%&+& C_3\int_\G\int_\G \left\{\nu_x\nabla_x^\top \left[\gamma_{k_1}(x,y) -\gamma_{k_2}(x,y)\right] M_yu(y)\right\} \cdot\ov{v}(x)ds_yds_x\nonumber\\
%&+& C_3\int_\G\int_\G \left\{\nabla_y\left[\gamma_{k_1}(x,y) -\gamma_{k_2}(x,y)\right]\nu_y^\top u(y)\right\}\cdot M_x\ov{v}(x)ds_yds_x,
%\en
%in which all the integrals are at most weakly-singular.
Here, the constants $C_i, i=1,2,3$ are given by
\ben
C_1=\frac{i\omega\eta\gamma(k_p^2+k_1^2)-k_1^2(k_1^2-q)(\lambda+2\mu)} {k_1^2-k_2^2},\\ C_2=\frac{i\omega\eta\gamma(k_p^2+k_2^2)-k_2^2(k_2^2-q)(\lambda+2\mu)} {k_1^2-k_2^2}
\enn
and
\ben
C_3=\frac{2\mu}{k_1^2-k_2^2}\left( \frac{i\omega\eta\gamma}{\lambda+2\mu} -k_2^2+q \right).
\enn
\end{theorem}
\begin{proof}
See \ref{appendex.c}.
\end{proof}

Next we investigate the terms $W_2p$ and $W_3u$. We have

\begin{theorem}
\label{main2}
The hyper-singular operators $W_2$ and $W_3$ can be expressed as
\be
\label{W2-1}
&\quad& W_2p(x) \nonumber\\
&=& -\frac{\gamma k_p^2\nu_x}{k_1^2-k_2^2}\int_\G \pa_{\nu_y}\left( \gamma_{k_1}(x,y)-\gamma_{k_2}(x,y)\right)p(y)ds_y\nonumber\\
&+& \frac{2\mu\gamma}{(k_1^2-k_2^2)(\lambda+2\mu)} M_x\int_\G M_yp(y)\nabla_y(\gamma_{k_1}(x,y)-\gamma_{k_2}(x,y))ds_y\nonumber\\
&+& \frac{2\mu\gamma}{(k_1^2-k_2^2)(\lambda+2\mu)}M_x \int_\G (k_1^2\gamma_{k_1}(x,y)-k_2^2\gamma_{k_2}(x,y))\nu_yp(y)ds_y,
\en
and
\be
\label{W3-1}
&\quad& W_3u(x) \nonumber\\
&=& -\frac{i\omega\eta k_p^2}{k_1^2-k_2^2}\int_\G\pa_{\nu_x}\left( \gamma_{k_1}(x,y)-\gamma_{k_2}(x,y)\right)\nu_y^\top u(y)ds_y\nonumber\\
&-& \frac{2i\mu\omega\eta}{(k_1^2-k_2^2)(\lambda+2\mu)} \int_\G M_x\nabla_x(\gamma_{k_1}(x,y)-\gamma_{k_2}(x,y))\cdot M_yu(y)ds_y\nonumber\\
&+& \frac{2i\mu\omega\eta}{(k_1^2-k_2^2)(\lambda+2\mu)}\int_\G (k_1^2\gamma_{k_1}(x,y)-k_2^2\gamma_{k_2}(x,y))\nu_x^\top M_yu(y)ds_y,
\en
respectively. %Moreover,
%\be
%\label{W2-2}
%&\quad&\langle W_2p,v\rangle \nonumber\\
%&=& -\frac{\gamma k_p^2}{k_1^2-k_2^2}\int_\G \int_\G \pa_{\nu_y}\left( \gamma_{k_1}(x,y)-\gamma_{k_2}(x,y)\right)p(y)\nu_x \cdot\ov{v}(x)ds_yds_x\nonumber\\
%&-& \frac{2\mu\gamma}{(k_1^2-k_2^2)(\lambda+2\mu)}\int_\G \int_\G M_yp(y)\nabla_x(\gamma_{k_1}(x,y)-\gamma_{k_2}(x,y))\cdot M_x\ov{v}(x)ds_yds_x\nonumber\\
%&+& \frac{2\mu\gamma}{(k_1^2-k_2^2)(\lambda+2\mu)} \int_\G \int_\G (k_1^2\gamma_{k_1}(x,y)-k_2^2\gamma_{k_2}(x,y))\nu_yp(y)\cdot M_x\ov{v}(x)ds_yds_x,
%\en
%\be
%\label{W3-2}
%&\quad& \langle W_3u,q\rangle \nonumber\\
%&=& -\frac{i\omega\eta k_p^2}{k_1^2-k_2^2}\int_\G\int_\G \pa_{\nu_x}\left( \gamma_{k_1}(x,y)-\gamma_{k_2}(x,y)\right)\nu_y^\top u(y)\ov{q}(x)ds_yds_x\nonumber\\
%&+& \frac{2i\mu\omega\eta}{(k_1^2-k_2^2)(\lambda+2\mu)} \int_\G \int_\G M_x\ov{q}(x)\nabla_x(\gamma_{k_1}(x,y)-\gamma_{k_2}(x,y))\cdot M_yu(y)ds_yds_x\nonumber\\
%&+& \frac{2i\mu\omega\eta}{(k_1^2-k_2^2)(\lambda+2\mu)}\int_\G \int_\G (k_1^2\gamma_{k_1}(x,y)-k_2^2\gamma_{k_2}(x,y))\nu_x^\top M_yu(y)\ov{q}(x)ds_yds_x,
%\en
%in which all the integrals are at most weakly-singular.
\end{theorem}

\begin{proof}
It follows that
\ben
&\quad& T(\pa_z,\nu_x)\pa_{\nu_y}E_{12}(z,y)-\gamma\nu_x\pa_{\nu_y}E_{22}(z,y)\\
&=& \frac{\gamma\nu_x}{k_1^2-k_2^2}\pa_{\nu_y}\left[ k_1^2\gamma_{k_1}(z,y)-k_2^2\gamma_{k_2}(z,y)+ (k_p^2-k_1^2)\gamma_{k_1}(z,y)-(k_p^2-k_2^2)\gamma_{k_2}(z,y)\right] \\
&-& \frac{2\mu\gamma}{(k_1^2-k_2^2)(\lambda+2\mu)},
\enn
and
\ben
&\quad& M_{z,x}\pa_{\nu_y}\nabla_z(\gamma_{k_1}(z,y)-\gamma_{k_2}(z,y))\\
&=& \frac{\gamma k_p^2\nu_x}{k_1^2-k_2^2}\pa_{\nu_y}\left( \gamma_{k_1}(z,y)-\gamma_{k_2}(z,y)\right) + \frac{2\mu\gamma}{(k_1^2-k_2^2)(\lambda+2\mu)} M_{z,x}\pa_{\nu_y}\nabla_y(\gamma_{k_1}(z,y)-\gamma_{k_2}(z,y))\\
&=& \frac{\gamma k_p^2\nu_x}{k_1^2-k_2^2}\pa_{\nu_y}\left( \gamma_{k_1}(z,y)-\gamma_{k_2}(z,y)\right)+ \frac{2\mu\gamma}{(k_1^2-k_2^2)(\lambda+2\mu)} M_{z,x}M_y\nabla_y(\gamma_{k_1}(z,y)-\gamma_{k_2}(z,y))\\
&-& \frac{2\mu\gamma\nu_y}{(k_1^2-k_2^2)(\lambda+2\mu)}M_{z,x} (k_1^2\gamma_{k_1}(z,y)-k_2^2\gamma_{k_2}(z,y)),
\enn
which further implies (\ref{W2-1}) by the Stokes formulas (\ref{stokes3}). The proof of (\ref{W3-1}) is similar and we omit it here.
\end{proof}

Finally, we investigate the term $W_4p$. From the results for acoustic wave (\ref{Wf2}), we immediately conclude that
\begin{theorem}
\label{main3}
The hyper-singular operator $W_4$ can be expressed as
\be
\label{W4-1}
&\quad& W_4 p(x) \nonumber\\
&=& \frac{1}{k_1^2-k_2^2}\int_\G \left(\nu_x\times\nabla_x [(k_p^2-k_1^2)\gamma_{k_1}(x,y)- (k_p^2-k_2^2)\gamma_{k_2}(x,y)] \right)\cdot (\nu_y\times\nabla_y p(y))ds_y \nonumber\\
&+& \frac{1}{k_1^2-k_2^2}\int_\G [k_1^2(k_p^2-k_1^2)\gamma_{k_1}(x,y)- k_2^2(k_p^2-k_2^2)\gamma_{k_2}(x,y)]\nu_x^\top \nu_y p(y)ds_y.
\en
%Moreover,
%\be
%\label{W4-2}
%&\quad&\langle W_4 p,q\rangle \nonumber\\
%&=& -\frac{k_p^2-k_1^2}{k_1^2-k_2^2}\int_\G\int_\G  \gamma_{k_1}(x,y) (\nu_y\times\nabla_y p(y))\cdot (\nu_x\times\nabla_x \ov{q}(x)) ds_yds_x \nonumber\\
%&-& \frac{k_p^2-k_2^2}{k_1^2-k_2^2}\int_\G\int_\G  \gamma_{k_2}(x,y) (\nu_y\times\nabla_y p(y))\cdot (\nu_x\times\nabla_x \ov{q}(x)) ds_yds_x \nonumber\\
%&+& \frac{1}{k_1^2-k_2^2}\int_\G\int_\G [k_1^2(k_p^2-k_1^2)\gamma_{k_1}(x,y)- k_2^2(k_p^2-k_2^2)\gamma_{k_2}(x,y)]\nu_x^\top \nu_y p(y)\ov{q}(x)ds_yds_x,
%\en
%in which all the integrals are at most weakly-singular.
\end{theorem}

\begin{remark}
It can be easily verified from the Stokes formulas of G\"unter derivatives that using the proposed regularized formulations, all the integrals in the corresponding weak forms of $W_su$ and $\widetilde{W}_sU$ are at most weakly-singular.
\end{remark}

\section{Numerical tests}
\label{sec:ne}

In this section, we present several numerical examples to demonstrate the accuracy of the  proposed scheme solving the exterior ESP and TESP.   We now take the ESP as the model to describe the method for numerical implementations.

\subsection{Numerical implementations}
\label{sec:ni}

Using the weak form of the regularized formulation (\ref{Ws1}), it follows that (\ref{coematrix1}) can be retreated as
\be
\label{coematrix11}
{\bf A}_h(k,j)&=& \rho\omega^2\int_{\G_h}\int_{\G_h} \gamma_{k_s}(x,y)(\nu_x\nu_y^\top-\nu_x^\top\nu_yI-J_{\nu_x,\nu_y})\psi_j(y) \psi_k(x)ds_yds_x\nonumber\\
&-& \rho\omega^2\int_{\G_h}\int_{\G_h} \gamma_{k_p}(x,y)\nu_x\nu_y^\top\psi_j(y) \psi_k(x)ds_yds_x\nonumber\\
&-& 2\mu\int_{\G_h}\int_{\G_h} M_x\psi_k(x) \nabla_y[\gamma_{k_s}(x,y)-\gamma_{k_p}(x,y)]\nu_y^\top \psi_j(y)\,ds_yds_x\nonumber\\
&+& \mu\int_{\G_h}\int_{\G_h} \gamma_{k_s}(x,y) \left(\nu_y\times\nabla_y\psi_j(y)\right)\cdot \left(\nu_x\times\nabla_x\psi_k(x)\right)ds_yds_x I\nonumber\\
&-& 2\mu\int_{\G_h}\int_{\G_h}\gamma_{k_s}(x,y) M_x\psi_k(x)M_y\psi_j(y)\,ds_yds_x\nonumber\\
&+& 4\mu^2\int_{\G_h}\int_{\G_h} M_x\psi_k(x) E(x,y)M_y\psi_j(y)\,ds_yds_x\nonumber\\
&+& 2\mu\int_{\G_h}\int_{\G_h} \nu_x\nabla_x^\top [\gamma_{k_s}(x,y)-\gamma_{k_p}(x,y)] M_y\psi_j(y)\psi_k(x)\,ds_yds_x\nonumber\\
&-& \mu \sum_{j,k,l=1}^3\int_{\G_h}\int_{\G_h} \gamma_{k_s}(x,y)M_{\psi_j,\psi_k}ds_yds_x,
\en
in which the entries of the matrix $M_{\psi_j,\psi_k}$ are given by
\ben
M_{\psi_j,\psi_k}(m,n)=\sum_{l=1}^3m_y^{lm}\psi_j(y)m_x^{nl}\psi_k(x),\quad m,n=1,2,3.
\enn
In (\ref{coematrix11}), all the integrals are at most weakly-singular. It can be obtained from decompositions that the weakly-singular kernels in (\ref{coematrix11}) are of types
\ben
\frac{1}{|x-y|},\quad \frac{(x-y)(x-y)^\top}{|x-y|^3}.
\enn

\label{sec:lcs}
\begin{figure}[htbp]
\centering
\includegraphics[scale=0.6]{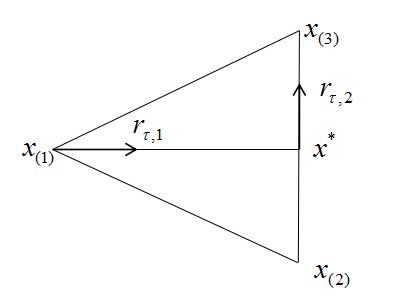}
\caption{Boundary element $\tau$.}
\label{Fig1}
\end{figure}

To compute the weakly singular integrals efficiently, we apply the special local coordinate system given in \cite{RS} to the boundary element $\tau$ with vertex $x_{(1)},x_{(2)},x_{(3)}$. Define the unit vector $r_{\tau,2}=(x_{(3)}-x_{(2)})/(|x_{(3)}-x_{(2)}|)$. Set $q_1=(x_{(2)}-x_{(1)})\cdot r_{\tau,2}$ and $q_2=(x_{(3)}-x_{(1)})\cdot r_{\tau,2}$. Then the intersection point $x^*$ can be determined by $x^*=x_{(3)}-q_2r_{\tau,2}$ or equivalently, $x^*=x_{(2)}-q_1r_{\tau,2}$. Define another unit vector $r_{\tau,1}=(x^*-x_{(1)})/(p_\tau)$, $p_\tau=|x^*-x_{(1)}|$ and the proportional coefficients $\alpha_i=q_i/s_\tau,\quad i=1,2$. Thus, the boundary element $\tau$ can be parameterized as
\ben
\tau=\{x=x(p_x,q_x)=x_{(1)}+p_xr_{\tau,1}+q_xr_{\tau,2}: 0<p_x<p_\tau,\; \alpha_1p_x<q_x<\alpha_2p_x\}.
\enn
Then for $x,y\in\tau$, $|x-y|^2=(p_x-p_y)^2+(q_x-q_y)^2$. The outward unit norma $\nu_\tau$ to the element $\tau$ is determined by $\nu_\tau=r_{\tau,1}\times r_{\tau,2}$. Moreover, the piecewise linear basis function $\psi_{(1)}$ for the vertex $x_{(1)}$ on $\tau$ can be formulated as $\psi_{(1)}(x)=\psi_{(1)}(x(p_x,q_x))=(p_\tau-p_x)/(p_\tau),  x\in\tau$, and $\nabla_x\psi_{(1)}(x)=-r_{\tau,1}/p_\tau$. In addition, using the above parametrisation we have
\ben
\int_\tau f(x)\,ds_x= \int_0^{p_\tau}\int_{\alpha_1p_x}^{\alpha_2p_x}f(x(p_x,q_x))\,dq_xdp_x,
\enn
or
\ben
\int_\tau f(x)\,ds_x &=& \int_{q_1}^0\int_{p_\tau-(q_1-q_x)/\alpha_1}^{p_\tau}f(x(p_x,q_x))\,dp_xdq_x\\
&+& \int_0^{q_2}\int_{p_\tau-(q_2-q_x)/\alpha_2}^{p_\tau} f(x(p_x,q_x))\,dp_xdq_x.
\enn

Now we present the main computing strategy of the numerical implementation. Set $\tau=\tau_i$. Corresponding to the piecewise linear basis function $\psi_{i_m}(x), m=1,2,3$ on $\tau$, set
\ben
x_{(n)}=x_{i_{B(m,n)}},\quad n=1,2,3,\quad B=\begin{bmatrix}
1 & 2 & 3\\
2 & 3 & 1\\
3 & 1 & 2
\end{bmatrix}.
\enn
Using this reorder strategy and the local coordinate system, $\psi_{i_m}(x)=\psi_{(1)}(x(p_x,q_x))$ on $\tau$. Therefore,
\ben
\nabla_x\psi_{i_m}(x)=-\frac{1}{p_\tau}r_{\tau,1},\quad M_x\psi_{i_m}(x)=-\frac{1}{p_\tau} (r_{\tau,1}\nu_\tau^\top-\nu_\tau r_{\tau,1}^\top),\quad x\in\tau,
\enn
are all constants.

The nonsingular integrals involved in (\ref{coematrix11}) can be approximated by Gaussian quadrature for triangular elements and we only need to consider the following integrals
\ben
I_1 &=& \int_{\tau_i}\int_{\tau_i} \frac{1}{|x-y|}\,ds_yds_x,\\
I_2 &=& \int_{\tau_i}\int_{\tau_i} \frac{1}{|x-y|}\psi_{i_m}(y)\psi_{i_n}(x)\,ds_yds_x,\quad m,n=1,2,3,\\
I_3 &=& \int_{\tau_i}\int_{\tau_i} \frac{(x-y)(x-y)^\top}{|x-y|^3}\,ds_yds_x,
\enn
which can be numerically computed following the steps described in \cite{RS} in a semi-analytic sense.

\subsection{Numerical examples}

In the numerical tests, the direct solver '$\setminus$' in Matlab is employed for solutions of the linear system (\ref{linearsys}). The impenetrable obstacle $\Om$ is set to be a unit ball (see Figure \ref{obstacle} (a)) or star-like (see Figure \ref{obstacle} (b)) with radial function
\ben
r(\theta,\phi)=\sqrt{0.8+0.5(\cos2\phi-1)(\cos4\theta-1)},\quad\theta\in[0,\pi], \;\phi\in[0,2\pi].
\enn
For these two obstacles, the origin $O$ is in $\Omega$. In our numerical tests, we first compute the unknown potentials $\varphi_h$ and $\Psi_h$ on $\Gamma_h$ by solving the variational equations (\ref{weak1}) and (\ref{weak2}), respectively and then put them into the solution representations (\ref{DirectBRF1}) and (\ref{DirectBRF2}) to get the numerical solutions $u_h$ and $U_h$ in $\Om^c$, i.e.,
\ben
u_h(x) &=& \int_{\Gamma_h}(T_y E(x,y))^\top \varphi_h(y)\,ds_y,\\
U_h(x) &=& \int_{\Gamma_h}(\widetilde{T}^*(\pa_y,\nu_y) \widetilde{E}^\top(x,y))^\top \Psi_h(y)\,ds_y.
\enn

\begin{figure}[htbp]
\centering
\begin{tabular}{ccc}
\includegraphics[scale=0.24]{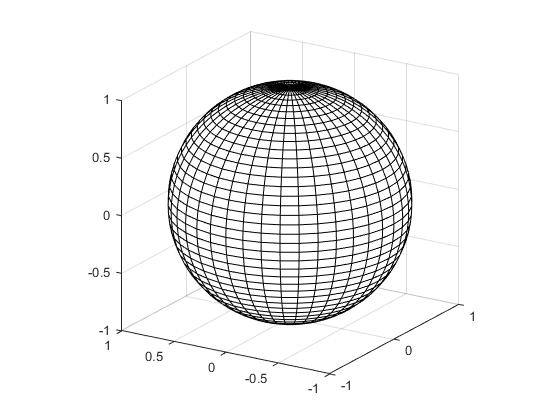} &
\includegraphics[scale=0.24]{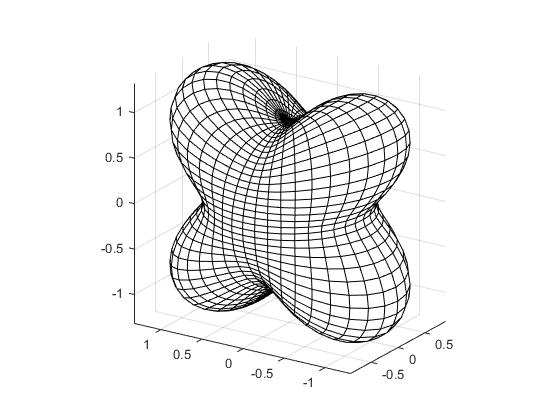} \\
(a) Obstacle I & (b) Obstacle II
\end{tabular}
\caption{Impenetrable obstacles to be considered in numerical tests.}
\label{obstacle}
\end{figure}

\subsubsection{Numerical examples for ESP}

Set $\omega=1$, $\rho=1$, $\lambda=2$, $\mu=1$. Let the exact solution be
\ben
u(x)=\nabla_x\left(\frac{e^{ik_p|x|}}{4\pi|x|}\right),\quad x\in\Omega^c.
\enn
Denote $\G_m:=\{x=(x_1,x_2,x_3)^\top\in\R^3: x_1=2\cos\theta,x_2=2,x_3=1.5\cos\theta,\theta\in[0,2\pi]\}$. Define the numerical error
\ben
\mbox{Error}:=\|u-u_h\|_{L^\infty(\G_m)^3}.
\enn
For simplicity, we use 'RP' and 'IP' to stand for 'real part' and 'imaginary part', respectively. The exact and numerical solutions on $\G_m$ are plotted in Figure \ref{Fig3} for Obstacle I with $h=0.1005$. We observe that the numerical solutions are in a perfect agreement with the exact ones from the qualitative point of view. In Table \ref{Table1}, we present the numerical errors $\mbox{Error}$ with respect to the meshsize $h$ which indicate the asymptotic convergence order $O(h^2)$. These results verify the accuracy of the regularized formulation for hyper-singular BIO $W_s$.

\begin{figure}[ht]
\centering
\begin{tabular}{ccc}
\includegraphics[scale=0.18]{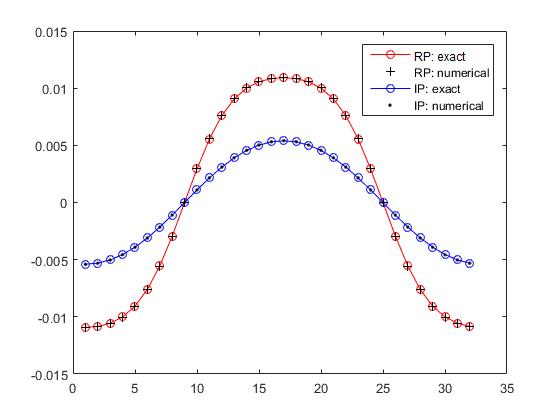} &
\includegraphics[scale=0.18]{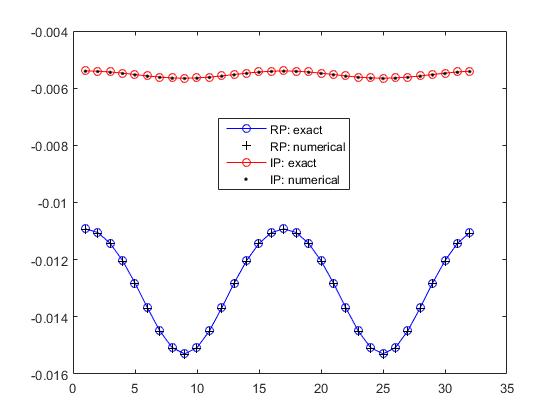} &
\includegraphics[scale=0.18]{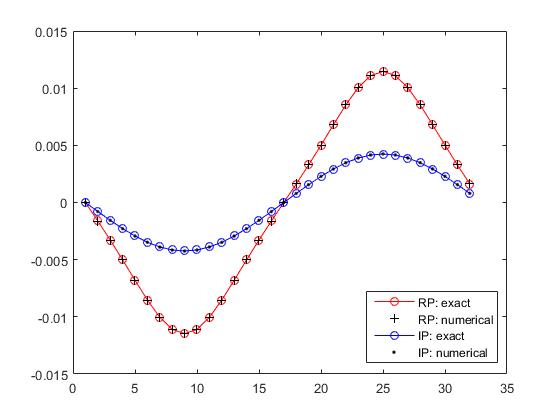} \\
(a) $u_1$ & (b) $u_2$ & (c) $u_3$
\end{tabular}
\caption{The real and imaginary parts of the exact and numerical solutions when $\Om$ is Obstacle I with $h=0.1005$.}
\label{Fig3}
\end{figure}

\begin{table}[ht]
\caption{Numerical errors $\mbox{Error}$ with respect to the meshsize $h$.}
\centering
\begin{tabular}{ccc}
\hline
$h$ & $\mbox{Error}$ & Order \\
\hline
0.4880 & 1.46E-3 & --    \\
0.3871 & 7.51E-4 & 2.87  \\
0.2668 & 2.95E-4 & 2.51  \\
0.1913 & 1.33E-4 & 2.39  \\
0.1005 & 3.01E-5 & 2.31  \\
\hline
\end{tabular}
\label{Table1}
\end{table}

Next, we consider the scattering of an incident plane wave $u^{in}$ taking the form
\ben
u^{in}=ik_pde^{ik_px\cdot d},\quad x\in\R^3,\quad d=(\sin\theta^{in}\cos\phi^{in}, \sin\theta^{in}\sin\phi^{in}, \cos\theta^{in})^\top\in\mathcal{S}^2.
\enn
by Obstacle II where $(\theta^{in},\phi^{in})$ is the incident direction. In this case, $f=-T(\pa,\nu)u^{in}$ on $\G$. We choose $\theta^{in}=\pi/2$ and $\phi^{in}=0$. The real and imaginary parts of the numerical solution $u_h$ on four unit spheres surrounding the obstacle is presented in Figure \ref{Fig5}.

\begin{figure}[ht]
\centering
\begin{tabular}{ccc}
\includegraphics[scale=0.18]{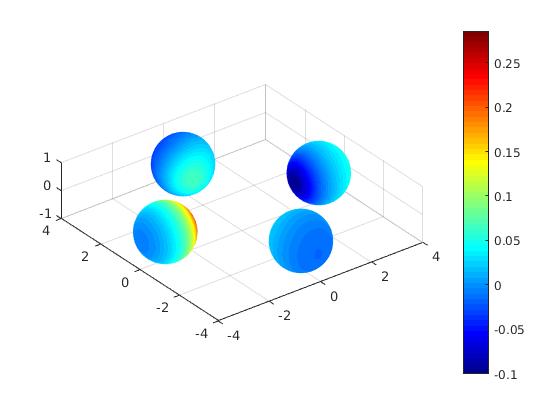} &
\includegraphics[scale=0.18]{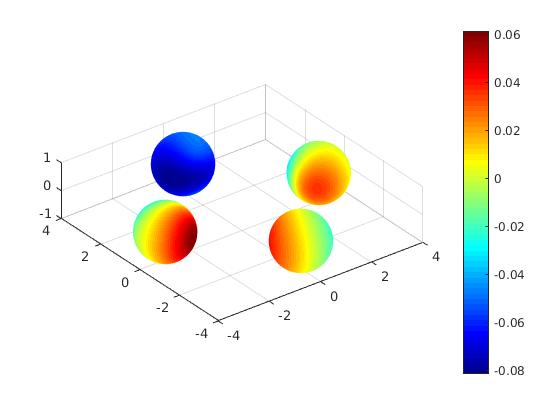} &
\includegraphics[scale=0.18]{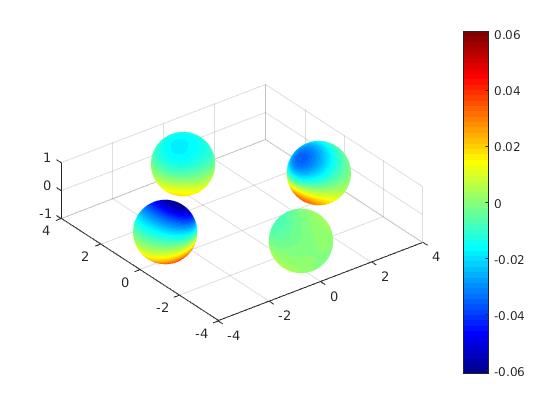} \\
(a) $\mbox{Re}(u_1)$ & (b) $\mbox{Re}(u_2)$ & (c) $\mbox{Re}(u_3)$ \\
\includegraphics[scale=0.18]{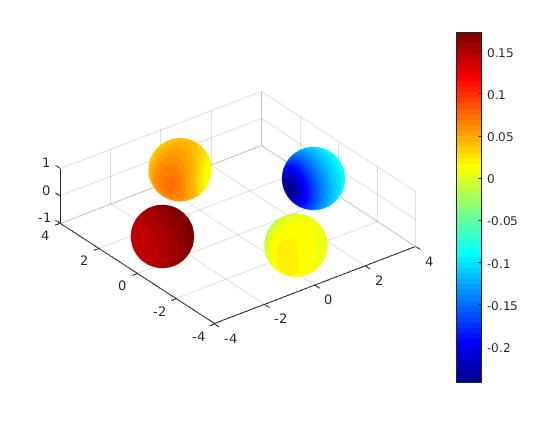} &
\includegraphics[scale=0.18]{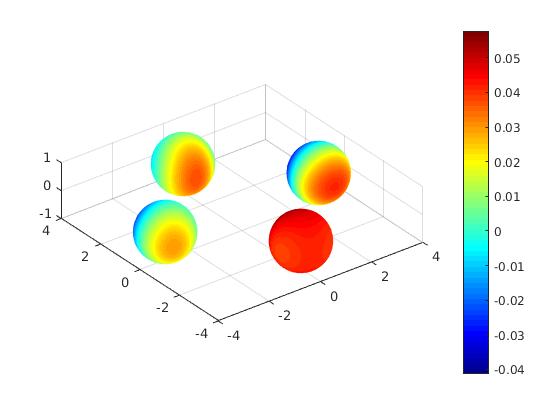} &
\includegraphics[scale=0.18]{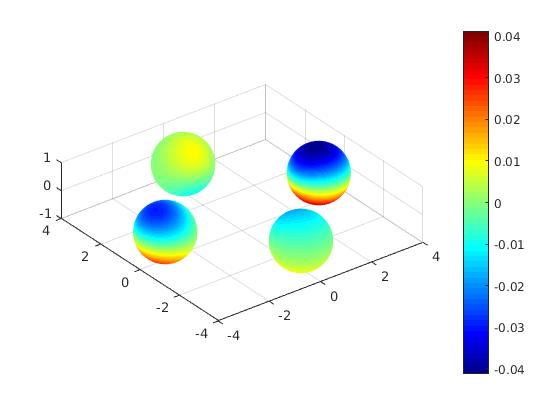} \\
(d) $\mbox{Im}(u_1)$ & (e) $\mbox{Im}(u_2)$ & (f) $\mbox{Im}(u_3)$
\end{tabular}
\caption{The real and imaginary parts of the numerical solutions of the scattering of plane incident wave for Obstacle II.}
\label{Fig5}
\end{figure}

\subsubsection{Numerical examples for TESP}

Choose $\omega=1$, $\rho=2$, $\lambda=1$, $\mu=1$, $\kappa=1$, $\eta=0.2$ and $\gamma=0.1$. The exact solution is set to be
\ben
u(x)=E_{12}(x,z),\quad p(x)=E_{22}(x,z),\quad x\in\Om^c,
\enn
and $z=(0,1,0.3,0.2)^\top\in\Omega$. Define the numerical error
\ben
\widetilde{\mbox{Error}}:=\|U-U_h\|_{L^\infty(\G_m)^4}.
\enn
We plot the exact and numerical solutions on $\G_m$ in Figure \ref{Fig6} for Obstacle I with $h=0.1005$. The numerical solutions are in a perfect agreement with the exact ones from the qualitative point of view. In Table \ref{Table2}, we present the numerical errors $\widetilde{\mbox{Error}}$ with respect to the meshsize $h$ which also indicate the convergence. These results verify the accuracy of the regularized formulation for hyper-singular BIO $\widetilde{W}_s$.

\begin{figure}[ht]
\centering
\begin{tabular}{cccc}
\includegraphics[scale=0.18]{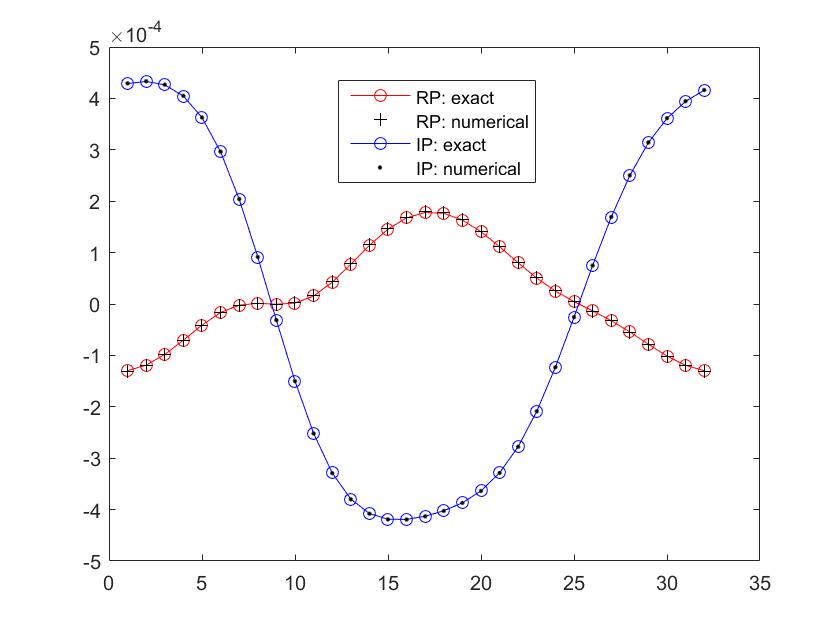} &
\includegraphics[scale=0.18]{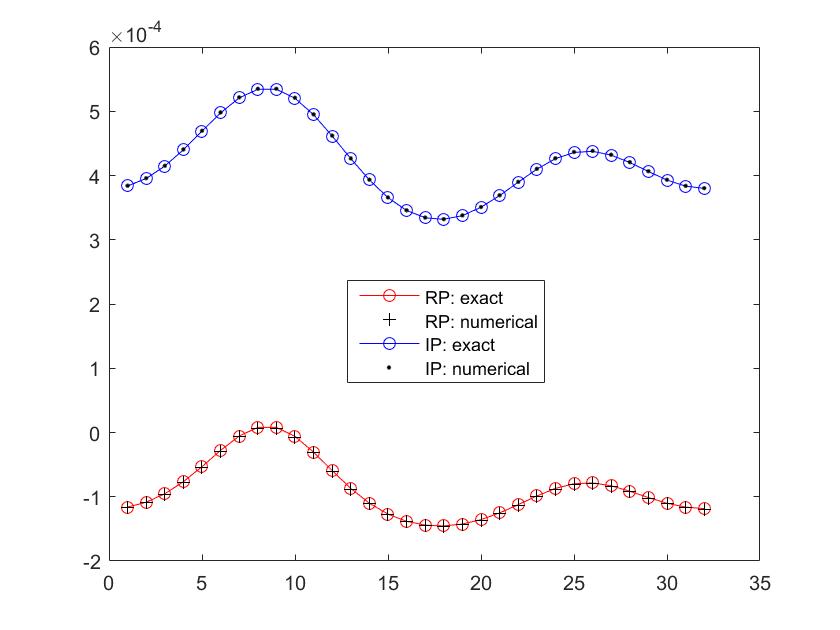} \\
(a) $u_1$ & (b) $u_2$ \\
\includegraphics[scale=0.18]{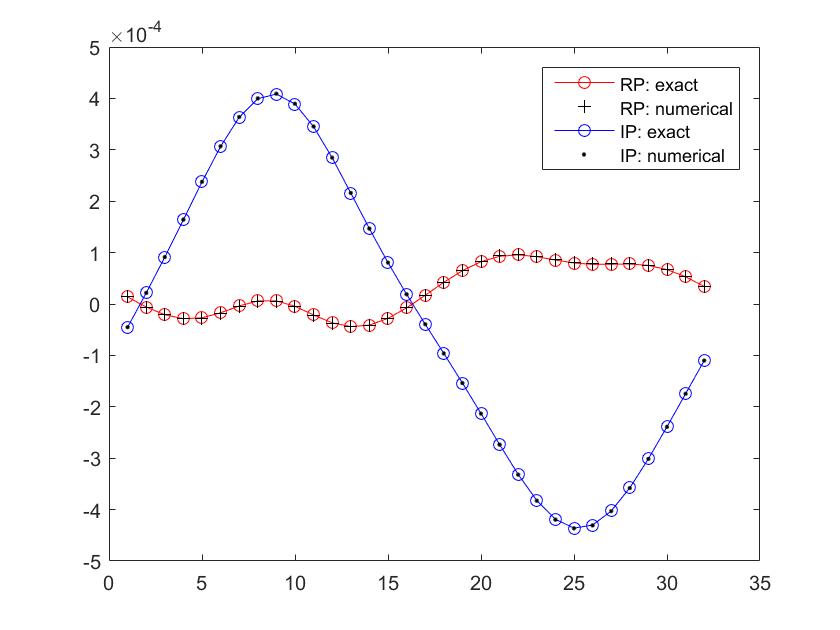} &
\includegraphics[scale=0.18]{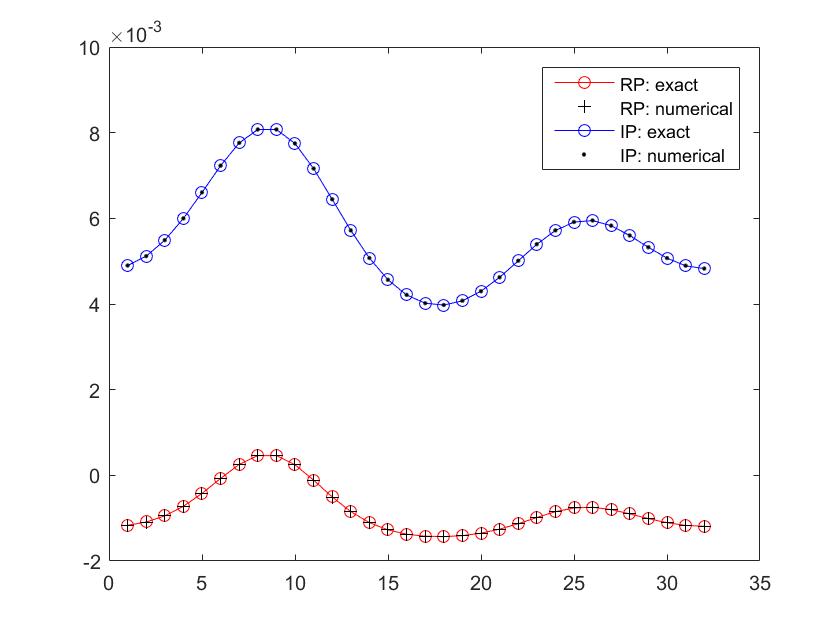} \\
(c) $u_3$ & (d) $p$
\end{tabular}
\caption{The real and imaginary parts of the exact and numerical solutions when $\Om$ is Obstacle I with $h=0.1005$.}
\label{Fig6}
\end{figure}

\begin{table}[ht]
\caption{Numerical errors $\widetilde{\mbox{Error}}$ with respect to the meshsize $h$.}
\centering
\begin{tabular}{ccc}
\hline
$h$ & $\widetilde{\mbox{Error}}$ & Order \\
\hline
0.4880 & 5.32E-4 & --    \\
0.3871 & 2.22E-4 & 3.77  \\
0.2668 & 1.07E-4 & 1.96  \\
0.1913 & 3.59E-5 & 3.28  \\
0.1005 & 2.61E-6 & 4.07  \\
\hline
\end{tabular}
\label{Table2}
\end{table}

Finally, we consider the scattering of an incident point source $U^{in}=({u^{in}}^\top,p^{in})^\top$ taking the form
\ben
u^{in}(x)=E_{12}(x,z),\quad p^{in}(x)=E_{22}(x,z),\quad x,z\in\Om^c,
\enn
by Obstacle II where $z$ is the location of point source. In this case, $F=-\widetilde{T}(\pa,\nu)U^{in}$ on $\G$. We choose $z=(0,0,2)^\top$. The real and imaginary parts of the numerical solutions $U_h$ on four unit spheres surrounding the obstacle is presented in Figure \ref{Fig8}.

\begin{figure}[ht]
\centering
\begin{tabular}{cccc}
\includegraphics[scale=0.09]{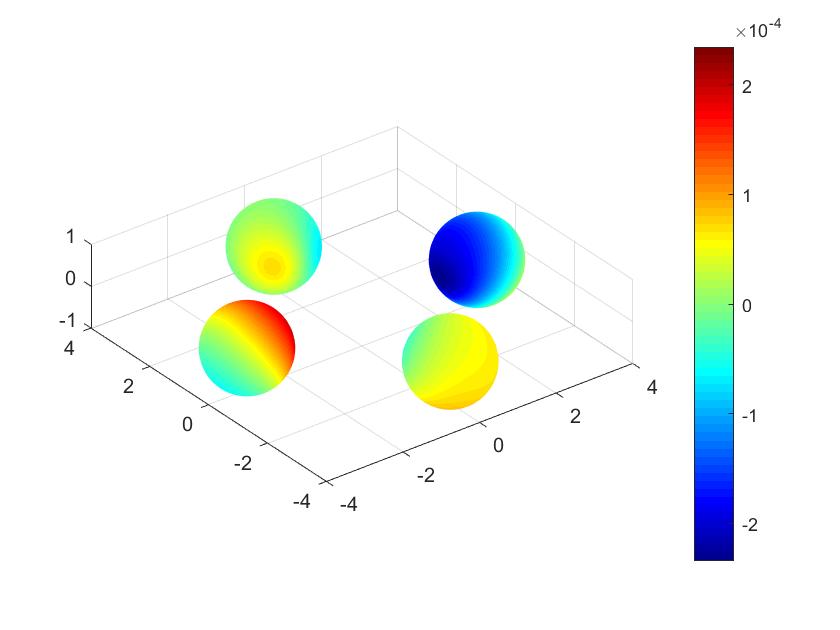} &
\includegraphics[scale=0.09]{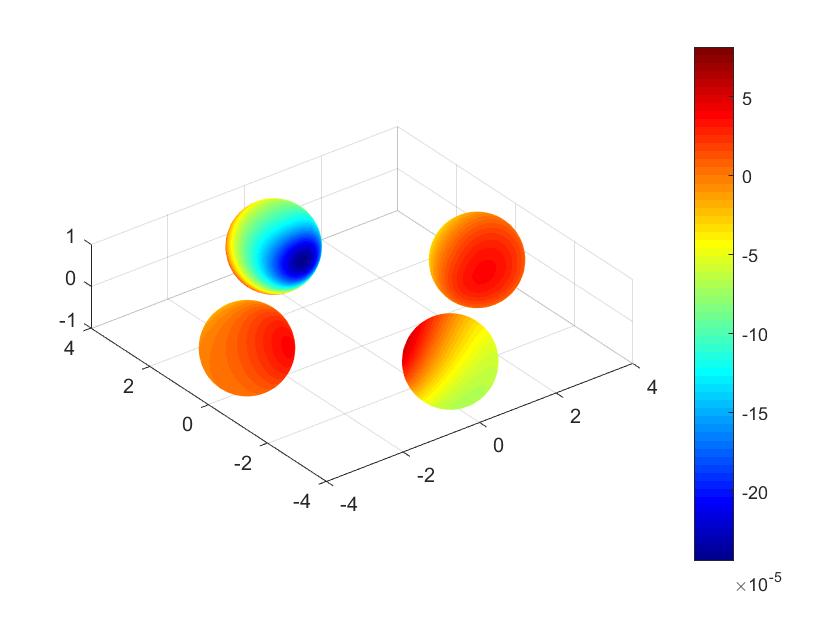} &
\includegraphics[scale=0.09]{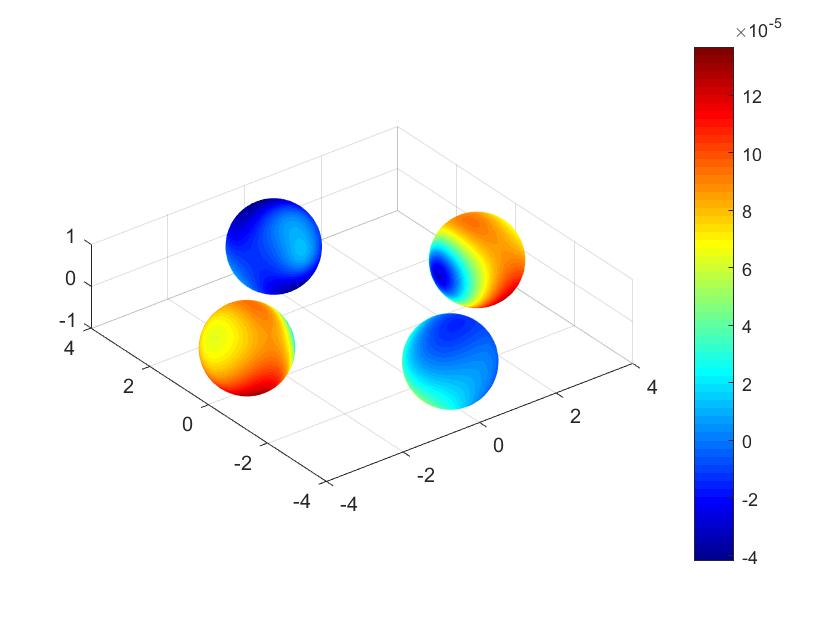} &
\includegraphics[scale=0.09]{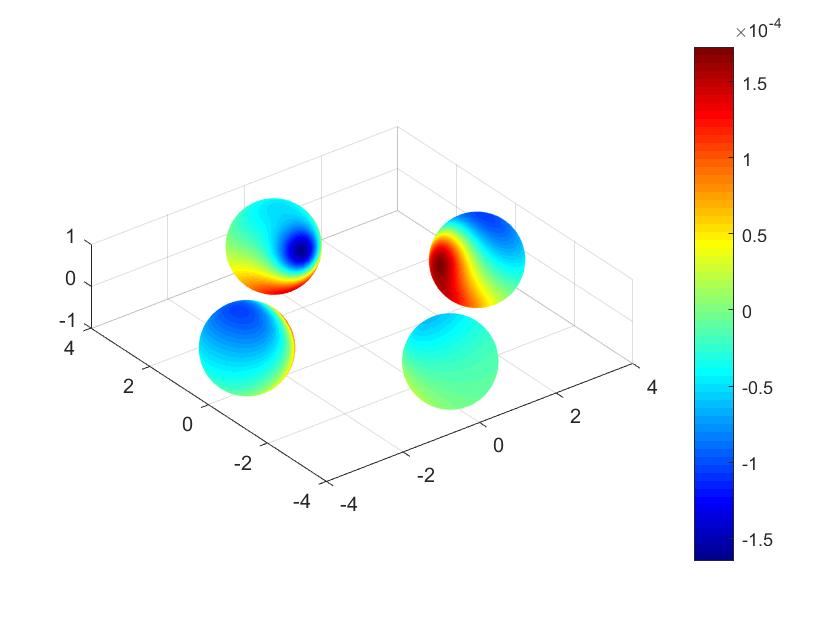} \\
(a) $\mbox{Re}(u_1)$ & (b) $\mbox{Re}(u_2)$ & (c) $\mbox{Re}(u_3)$ & (d) $\mbox{Re}(p)$ \\
\includegraphics[scale=0.09]{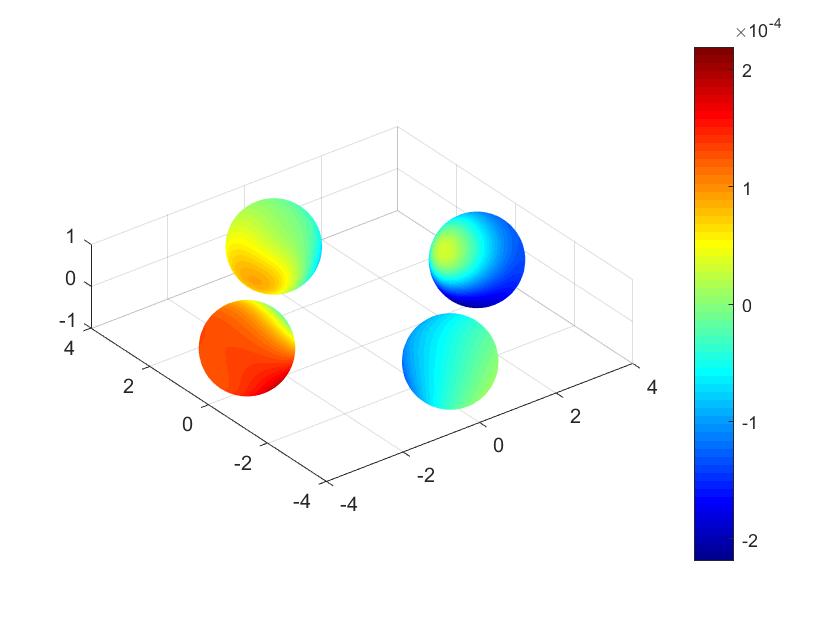} &
\includegraphics[scale=0.09]{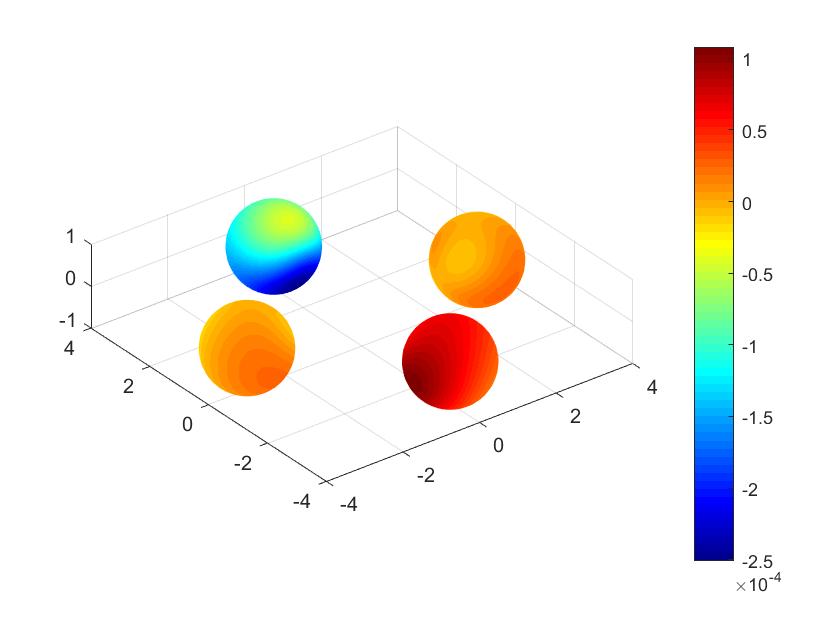} &
\includegraphics[scale=0.09]{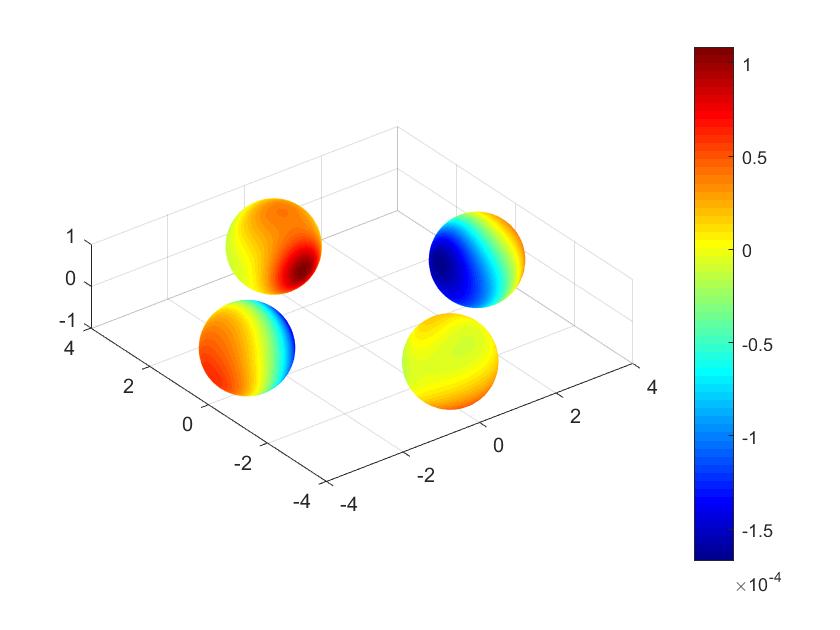} &
\includegraphics[scale=0.09]{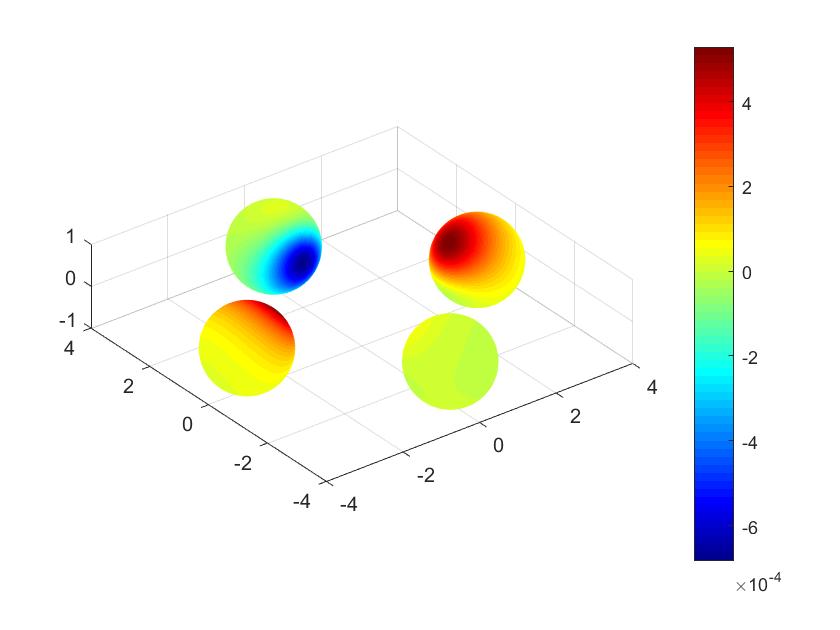} \\
(e) $\mbox{Im}(u_1)$ & (f) $\mbox{Im}(u_2)$ & (g) $\mbox{Im}(u_3)$ & (h) $\mbox{Im}(p)$
\end{tabular}
\caption{The real and imaginary parts of the numerical solutions of the scattering of point source for Obstacle II.}
\label{Fig8}
\end{figure}
%%%%%%%%

\appendix
\section{Proof of Theorem \ref{main}}
\label{appendex.a}
We know from (\ref{TyExy}) that
\ben
\mathcal{D}_su(z)=-f_1(z)+f_2(z)+f_3(z),
\enn
where
\ben
f_1(z) &=& \int_{\Gamma} \nabla_y [\gamma_{k_s}(z,y)-\gamma_{k_p}(z,y)]\nu_y^\top u(y)ds_y,\\
f_2(z) &=& \int_{\Gamma}  \pa_{\nu_y}\gamma _{k_s}(z,y)u(y)ds_y,\\
f_3(z) &=& \int_{\Gamma} [2\mu E(z,y)-\gamma_{k_s}(z,y)I] M_yu(y)ds_y.
\enn
Note that
\be
\label{Wsproof1}
W_su(x)= \lim_{z\rightarrow x\in\Gamma,z\notin\Gamma} (g_1(z)-g_2(z)-g_3(z)),
\en
where
\ben
g_i(z) = \mu\nu_x\cdot\nabla_z f_i(z) +(\lambda+\mu)\nu_x(\nabla_z \cdot f_i(z)+\mu M_{z,x}f_i(z).
\enn
We obtain from (\ref{Tgrad}) that
\be
\label{Wsproof2}
g_1(z)&=& (\lambda+2\mu) \int_\Gamma [k_s^2\gamma_{k_s}(z,y)-k_p^2\gamma_{k_p}(z,y)] \nu_x\nu_y^\top u(y)ds_y \nonumber\\
&+&  2\mu\int_\Gamma M_{z,x}\nabla_y[\gamma_{k_s}(z,y)-\gamma_{k_p}(z,y)]\nu_y^\top u(y)ds_y.
\en
From (\ref{Wf2}) we can obtain that
\be
\label{Wsproof3}
g_2(z)&=& \mu\int_\Gamma (\nu_x\cdot\nabla_z)\pa_{\nu_y}\gamma_{k_s}(z,y)u(y)ds_y +(\lambda+\mu)\int_\Gamma \nu_x\nabla_z^\top \pa_{\nu_y}\gamma_{k_s}(z,y) u(y)ds_y\nonumber\\
&+& \mu\int_\Gamma M_{z,x}\pa_{\nu_y}\gamma_{k_s}(z,y)u(y)ds_y\nonumber\\
&=& \mu\int_\Gamma \left(\nu_x\times\nabla_z\gamma_{k_s}(z,y)\right)\cdot \left(\nu_y\times\nabla_yu(y)\right)ds_y + \mu k_s^2\int_\Gamma \gamma_{k_s}(z,y)\nu_x^\top\nu_yu(y)ds_y\nonumber\\
&+& (\lambda+\mu)\int_\Gamma \nu_x\nabla_z^\top \pa_{\nu_y}\gamma_{k_s}(z,y) u(y)ds_y +\mu\int_\Gamma M_{z,x}\pa_{\nu_y}\gamma_{k_s}(z,y)u(y)ds_y
\en
For $g_3(z)$, we know from (\ref{TxExy}) that
\be
\label{Wsproof4}
&\quad& g_3(z)\nonumber\\
&=& 2\mu\int_\Gamma (\nu_x\cdot\nabla_z)\gamma_{k_s}(z,y) M_yu(y)ds_y \nonumber\\
&-& 2\mu\int_\Gamma \nu_x\nabla_z^\top [\gamma_{k_s}(z,y)-\gamma_{k_p}(z,y)]M_yu(y)ds_y\nonumber\\
&+& 4\mu^2\int_\Gamma M_{z,x}E(z,y)M_yu(y)ds_y -2\mu\int_\Gamma M_{z,x}\gamma_{k_s}(z,y)M_yu(y)ds_y \nonumber\\
&-& \mu\int_\Gamma (\nu_x\cdot\nabla_z)\gamma_{k_s}(z,y) M_yu(y)ds_y -(\lambda+\mu)\int_\Gamma \nu_x\nabla_z^\top\gamma_{k_s}(z,y)M_yu(y)ds_y\nonumber\\
&-& \mu\int_\Gamma M_{z,x}\gamma_{k_s}(z,y)M_yu(y)ds_y\nonumber\\
&=& \mu\int_\Gamma (\nu_x\cdot\nabla_z)\gamma_{k_s}(z,y) M_yu(y)ds_y -3\mu\int_\Gamma M_{z,x}\gamma_{k_s}(z,y)M_yu(y)ds_y\nonumber\\
&+& 4\mu^2\int_\Gamma M_{z,x}E(z,y)M_yu(y)ds_y \nonumber\\
&-& 2\mu\int_\Gamma \nu_x\nabla_z^\top [\gamma_{k_s}(z,y)-\gamma_{k_p}(z,y)]M_yu(y)ds_y\nonumber\\
&-& (\lambda+\mu)\int_\Gamma \nu_x\nabla_z^\top\gamma_{k_s}(z,y)M_yu(y)ds_y.
\en
Therefore, (\ref{Wsproof2})-(\ref{Wsproof4}) yields
\be
\label{Wsproof5}
&\quad& g_1(z)-g_2(z)-g_3(z) \nonumber\\
&=& (\lambda+2\mu) \int_\Gamma [k_s^2\gamma_{k_s}(z,y)-k_p^2\gamma_{k_p}(z,y)] \nu_x\nu_y^\top u(y)ds_y\nonumber\\
&+& 2\mu\int_\Gamma M_{z,x}\nabla_y[\gamma_{k_s}(z,y)-\gamma_{k_p}(z,y)]\nu_y^\top u(y)ds_y\nonumber\\
&-& \mu\int_\Gamma \left(\nu_x\times\nabla_z\gamma_{k_s}(z,y)\right)\cdot \left(\nu_y\times\nabla_yu(y)\right)ds_y - \mu k_s^2\int_\Gamma \gamma_{k_s}(z,y)\nu_x^\top\nu_yu(y)ds_y\nonumber\\
&+& 3\mu\int_\Gamma M_{z,x}\gamma_{k_s}(z,y)M_yu(y)ds_y -4\mu^2\int_\Gamma M_{z,x}E(z,y)M_yu(y)ds_y\nonumber\\
&+& 2\mu\int_\Gamma \nu_x\nabla_z^\top [\gamma_{k_s}(z,y)-\gamma_{k_p}(z,y)]M_yu(y)ds_y -\mu h_1(z)-(\lambda+\mu)h_2(z),
\en
where
\ben
h_1(z) &=& \int_\Gamma \left[ M_{z,x}\pa_{\nu_y}\gamma_{k_s}(z,y)u(y)+ (\nu_x\cdot\nabla_z)\gamma_{k_s}(z,y)M_yu(y)\right]ds_y\\
h_2(z) &=& \int_\Gamma \nu_x\left[\nabla_z^\top \pa_{\nu_y}\gamma_{k_s}(z,y) u(y)- \nabla_z^\top\gamma_{k_s}(z,y)M_yu(y)\right]ds_y.
\enn
Note that for $i,j=1,2,3$,
\ben
\sum_{l=1}^3 \left(m_y^{il}m_{z,x}^{lj}-m_{z,x}^{il}m_y^{lj}\right)
&=& (\nu_y^i\nu_x^j-\nu_x^i\nu_y^j)\Delta_z+ m^{ij}_{z,x}\pa_{\nu_y} -m^{ij}_y(\nu_x\cdot\nabla_z).
\enn
We conclude that
\ben
h_1(z) &=& \int_\Gamma \left[ M_{z,x}\pa_{\nu_y}\gamma_{k_s}(z,y)u(y)+ (\nu_x\cdot\nabla_z)\gamma_{k_s}(z,y)M_yu(y)\right]ds_y\nonumber\\
&=& \int_\Gamma \left[ M_{z,x}\pa_{\nu_y}\gamma_{k_s}(z,y)u(y)- M_y(\nu_x\cdot\nabla_z)\gamma_{k_s}(z,y)u(y)\right]ds_y\nonumber\\
&=& \int_\Gamma \left[ M_yM_{z,x}- M_{z,x}M_y\right]\gamma_{k_s}(z,y)u(y)ds_y\nonumber\\
&+& k_s^2\int_\Gamma \gamma_{k_s}(z,y)J(\nu_x,\nu_y)u(y)ds_y.
\enn
We obtain from the Stokes formula (\ref{stokes1}) that
\ben
&\quad& \lim_{z\rightarrow x\in\Gamma,z\notin\Gamma}\int_\Gamma M_yM_{z,x}\gamma_{k_s}(z,y)u(y)ds_y\\
&=& \left\{ \sum_{k,l=1}^3m_{jk}^ym_{kl}^x\gamma_{k_s}(x,y)u_l(y)ds_y \right\}_{j=1}^3\\
&=& \left\{ \sum_{k,l=1}^3m_{kl}^x\gamma_{k_s}(x,y)m_{kj}^yu_l(y)ds_y \right\}_{j=1}^3.
\enn
On the other hand,
\ben
\int_\Gamma M_{z,x}M_y\gamma_{k_s}(z,y)u(y)ds_y &=& -\int_\Gamma M_{z,x}\gamma_{k_s}(z,y)M_yu(y)ds_y.
\enn
Thus,
\be
\label{Wsproof6}
\lim_{z\rightarrow x\in\Gamma,z\notin\Gamma} h_1(z) &=& \left\{ \sum_{k,l=1}^3m_{kl}^x\gamma_{k_s}(x,y)m_{kj}^yu_l(y)ds_y \right\}_{j=1}^3\nonumber\\
&+& \int_\Gamma M_x\gamma_{k_s}(x,y)M_yu(y)ds_y\nonumber\\
&+& k_s^2\int_\Gamma \gamma_{k_s}(z,y)J(\nu_x,\nu_y)u(y)ds_y.
\en
Finally, since
\ben
\nu_x\int_\Gamma \nabla_z^\top\gamma_{k_s}(z,y)M_yu(y)ds_y =\nu_x\int_\Gamma \left[M_y\nabla_z\gamma_{k_s}(z,y)\right] \cdot u(y)ds_y
\enn
we have
\be
\label{Wsproof7}
h_2(z) &=& \nu_x\int_\Gamma \left[ \nabla_z \pa_{\nu_y}\gamma_{k_s}(z,y)- M_y\nabla_z\gamma_{k_s}(z,y) \right]\cdot u(y)ds_y\nonumber\\
&=& -\nu_x\int_\Gamma \Delta_z\gamma_{k_s}(z,y)\nu_y^\top u(y)ds_y\nonumber\\
&=& k_s^2\int_\Gamma \gamma_{k_s}(z,y)\nu_x\nu_y^\top u(y)ds_y.
\en
We complete the proof of (\ref{Ws1}) by a combination of (\ref{Wsproof1}) and (\ref{Wsproof5})-(\ref{Wsproof7}). %The weak form of $W_su$ (\ref{Ws2}) follows immediately from (\ref{Ws1}) and the Stokes formulas given in Theorem \ref{GP} (iii).

\section{Proof of Lemma \ref{Tlemma1}}
\label{appendex.b}
For some matrix $A$ or vector $B$, we denote $(A)_{ij}$ and $(B)_i$ their Cartesian components, respectively. Let
\ben
R_1=\gamma_{k_s}-\frac{k_p^2-k_2^2}{k_1^2-k_2^2}\gamma_{k_1}+ \frac{k_p^2-k_1^2}{k_1^2-k_2^2}\gamma_{k_2}.
\enn
Then we have
\be
\label{Tlemma1-1}
(\nabla_x\cdot E_{11})_i &=& \frac{1}{\mu}\partial_{x_i}\gamma_{k_s} +\frac{1}{\rho\omega^2} \sum_{j=1}^d\partial_{x_i}\partial_{x_j}^2R_1 \nonumber\\
&=& \partial_{x_i}\left( \frac{1}{\mu}\gamma_{k_s}+\frac{1}{\rho\omega^2} \Delta_x R_1 \right),
\en
\be
\label{Tlemma1-2}
(\partial_{\nu_x}E_{11})_{ij}= \frac{1}{\mu}\partial_{\nu_x}\gamma_{k_s}\delta_{ij} +\frac{1}{\rho\omega^2}\sum_{l=1}^d \nu_x^l\partial_{x_l}\partial_{x_i}\partial_{x_j} R_1,
\en
and
\be
\label{Tlemma1-3}
&\quad& (M_xE_{11})_{ij} \nonumber\\
&=& \frac{1}{\mu}M_x\gamma_{k_s}+ \frac{1}{\rho\omega^2} \sum_{l=1}^d (\partial_{x_i}\nu_x^l-\partial_{x_l}\nu_x^i) \partial_{x_l}\partial_{x_j}R_1\nonumber\\
&=& \frac{1}{\mu}M_x\gamma_{k_s} +\frac{1}{\rho\omega^2}\sum_{l=1}^d \nu_x^l\partial_{x_l}\partial_{x_i}\partial_{x_j} R_1 -\frac{1}{\rho\omega^2}\nu_x^i\pa_{x_j}\Delta_x R_1.
\en
Therefore, from (\ref{Tform2}) and (\ref{Tlemma1-1})-(\ref{Tlemma1-3}) we have
\ben
&\quad&(T(\pa_x,\nu_x)E_{11}(x,y))_{ij}\\
&=& (\lambda+\mu)\nu_x^i(\nabla_x\cdot E_{11})_j + \mu(\pa_{\nu_x}E_{11})_{ij} + \mu (M_xE)_{ij}\\
&=& \nu_x^i\pa_{x_j}\left( \frac{\lambda+\mu}{\mu}\gamma_{k_s}+ \frac{\lambda+2\mu}{\rho\omega^2} \Delta_x R_1 \right) +\partial_{\nu_x}\gamma_{k_s}\delta_{ij}+ (M_x(2\mu E_{11}-\gamma_{k_s}))_{ij}.
\enn
Note that
\ben
\Delta_x R_1 &=& -k_s^2\gamma_{k_s}+ \frac{(k_p^2-k_2^2)k_1^2}{k_1^2-k_2^2}\gamma_{k_1}- \frac{(k_p^2-k_1^2)k_2^2}{k_1^2-k_2^2}\gamma_{k_2}\\
&=& -k_s^2\gamma_{k_s} +\frac{(k_1^2-q)k_p^2}{k_1^2-k_2^2}\gamma_{k_1}- \frac{(k_2^2-q)k_p^2}{k_1^2-k_2^2}\gamma_{k_2}\\
&=& -k_s^2\gamma_{k_s}+ k_p^2\gamma_{k_1} +\frac{(k_2^2-q)k_p^2}{k_1^2-k_2^2} (\gamma_{k_1}-\gamma_{k_2}).
\enn
Hence,
\ben
(T(\pa_x,\nu_x)E_{11})_{ij} &=& -\nu_x^i\pa_{x_j} (\gamma_{k_s}-\gamma_{k_1}) + \frac{k_2^2-q}{k_1^2-k_2^2} \nu_x^i\pa_{x_j}(\gamma_{k_1}-\gamma_{k_2}) \\
&+& \partial_{\nu_x}\gamma_{k_s}\delta_{ij}+ (M_x(2\mu E_{11}-\gamma_{k_s}))_{ij}
\enn
which completes the proof of (\ref{TxE11}). The proof of (\ref{TyE11}) follows in similar way, and we skip it here.

\section{Proof of Theorem \ref{main1}}
\label{appendex.c}
Following the same steps in \ref{appendex.a} we can obtain that
\be
\label{W1proof8}
&\quad& -\lim_{z\rightarrow x\in\G,z\notin\G}T(\pa_z,\nu_x)\int_\G (T(\pa_y,\nu_y)E_{11}(z,y))^\top u(y)ds_y\nonumber\\
&=& \rho\omega^2\int_\G \gamma_{k_s}(x,y) (\nu_x\nu_y^\top-\nu_x^\top\nu_yI-J_{\nu_x,\nu_y})u(y)ds_y\nonumber\\
&-& \frac{k_1^2(k_1^2-q)(\lambda+2\mu)}{k_1^2-k_2^2}\int_\G  \gamma_{k_1}(x,y) \nu_x\nu_y^\top u(y)ds_y\nonumber\\
&+& \frac{k_2^2(k_2^2-q)(\lambda+2\mu)}{k_1^2-k_2^2}\int_\G\gamma_{k_2}(x,y) \nu_x\nu_y^\top u(y)ds_y\nonumber\\
&-& \mu\int_\Gamma \left(\nu_x\times\nabla_x\gamma_{k_s}(x,y)\right)\cdot \left(\nu_y\times\nabla_yu(y)\right)ds_y-4\mu^2\int_\Gamma M_xE(x,y)M_yu(y)ds_y\nonumber\\
&+& 2\mu\int_\Gamma M_x\gamma_{k_s}(x,y)M_yu(y)ds_y -\mu\left\{ \sum_{k,l=1}^3\int_\G m_x^{kl}\gamma_{k_s}(x,y) m_y^{kj}u_l(y)ds_y \right\}_{j=1}^3 \nonumber\\
&+& 2\mu\int_\G \nu_x\nabla_x^\top \left[\gamma_{k_s}(x,y) -\gamma_{k_1}(x,y)\right] M_yu(y)ds_y \nonumber\\
&+& 2\mu\int_\G M_x\nabla_y \left[\gamma_{k_s}(x,y) -\gamma_{k_1}(x,y)\right]\nu_y^\top u(y)ds_y\nonumber\\
&-& \frac{2\mu(k_2^2-q)}{k_1^2-k_2^2}\int_\G \nu_x\nabla_x^\top \left[\gamma_{k_1}(x,y) -\gamma_{k_2}(x,y)\right] M_yu(y)ds_y \nonumber\\
&-& \frac{2\mu(k_2^2-q)}{k_1^2-k_2^2}\int_\G M_x\nabla_y \left[\gamma_{k_1}(x,y) -\gamma_{k_2}(x,y)\right]\nu_y^\top u(y)ds_y.
\en
On the other hand, we have
\ben
T(\pa_z,\nu_x)E_{12}(z,y) &=& \frac{\gamma}{k_1^2-k_2^2} \nu_x\left[k_1^2\gamma_{k_1}(z,y)-k_2^2\gamma_{k_1}(z,y)\right]\nonumber\\
&-& \frac{2\mu\gamma}{(k_1^2-k_2^2)(\lambda+2\mu)} M_{z,x}\nabla_z\left[\gamma_{k_1}(z,y) -\gamma_{k_2}(z,y)\right] ,
\enn
and
\ben
T(\pa_y,\nu_y)E_{21}(z,y) &=& \frac{i\omega\eta}{k_1^2-k_2^2} \nu_y\left[k_1^2\gamma_{k_1}(z,y)-k_2^2\gamma_{k_1}(z,y)\right]\nonumber\\
&-& \frac{2i\mu\omega\eta}{(k_1^2-k_2^2)(\lambda+2\mu)} M_y\nabla_y\left[\gamma_{k_1}(z,y) -\gamma_{k_2}(z,y)\right].
\enn
Then we have
\be
\label{W1proof9}
&\quad& \lim_{z\rightarrow x\in\G,z\notin\G}\int_\G [i\omega\eta T(\pa_z,\nu_x)E_{12}(z,y)\nu_y^\top +\gamma\nu_x(T(\pa_y,\nu_y)E_{21}(z,y))^\top \nonumber\\
&\quad& \quad\quad\quad\quad\quad\quad\quad\quad\quad -i\omega\eta\gamma\nu_x\nu_y^\top E_{22}(z,y)]u(y)ds_y\nonumber\\
&=& \frac{i\omega\eta\gamma}{k_1^2-k_2^2}\int_\G [(k_p^2+k_1^2)\gamma_{k_1}(x,y)- (k_p^2+k_2^2)\gamma_{k_2}(x,y)]\nu_x\nu_y^\top u(y)ds_y\nonumber\\
&+& \frac{2i\mu\omega\eta\gamma}{(k_1^2-k_2^2)(\lambda+2\mu)} \int_\G M_x\nabla_y\left[\gamma_{k_1}(x,y) -\gamma_{k_2}(x,y)\right]\nu_y^\top u(y)ds_y\nonumber\\
&+& \frac{2i\mu\omega\eta\gamma}{(k_1^2-k_2^2)(\lambda+2\mu)} \int_\G
\nu_x\nabla_x^\top\left[\gamma_{k_1}(x,y) -\gamma_{k_2}(x,y)\right] M_yu(y)ds_y.
\en
Then (\ref{W1-1}) can be proved by combining (\ref{W1proof8}) and (\ref{W1proof9}).% and the weak form of $W_1u$ (\ref{W1-2}) follows immediately from (\ref{W1-1}) and the Stokes formulas given in Proposition \ref{GP} (iii).

\section*{Acknowledgments}
The work of G. Bao is supported in part by a NSFC Innovative Group Fund (No.11621101), an Integrated Project of the Major Research Plan of NSFC (No. 91630309), and an NSFC A3 Project (No. 11421110002). The work of L. Xu is partially supported by a Key Project of the Major Research Plan of NSFC (No. 91630205), and a NSFC Grant (No. 11771068).

\end{document}